\documentclass[12pt]{article}
\usepackage[shortlabels]{enumitem}
\usepackage{amsmath}
\usepackage{mathrsfs}
\usepackage{amsfonts}
\usepackage{array,color}
\usepackage{abstract}
\usepackage[left=1in, right=1in, bottom=0.60in,includefoot]{geometry}
\usepackage{graphicx}
\usepackage{indentfirst}
\usepackage{amsthm}
\usepackage{mathrsfs}
\usepackage{makeidx}
\usepackage{url}
\usepackage{dsfont}
\usepackage{latexsym}
\usepackage{amsmath}
\usepackage{amssymb}
\usepackage{color}
\usepackage{hyperref}
\usepackage{tikz}
\usepackage{tikz-cd}
\usepackage{MnSymbol}
\newcommand{\ca}{\curvearrowright}
\newcommand{\la}{\langle}
\newcommand{\ra}{\rangle}
\newcommand{\El}{\mathcal{L}}
\newcommand{\subscript}[2]{$#1 _ #2$}

\begin{document}
	\newtheorem{Lemma}{Lemma}
	\theoremstyle{plain}
	\newtheorem{theorem}{Theorem~}[section]
	\newtheorem{main}{Main Theorem~}
	\newtheorem{lemma}[theorem]{Lemma~}
	\newtheorem{assumption}[theorem]{Assumption~}
	\newtheorem{proposition}[theorem]{Proposition~}
	\newtheorem{corollary}[theorem]{Corollary~}
	\newtheorem{definition}[theorem]{Definition~}
	\newtheorem{defi}[theorem]{Definition~}
	\newtheorem{notation}[theorem]{Notation~}
	\newtheorem{example}[theorem]{Example~}
	\newtheorem{remark}{Remark~}
	\newtheorem*{cor}{Corollary~}
	\newtheorem*{question}{Question}
	\newtheorem*{claim}{Claim}
	\newtheorem*{conjecture}{Conjecture~}
	\newtheorem*{fact}{Fact~}
	\renewcommand{\proofname}{\bf Proof}
    \newcommand{\email}{Email: }

\title{Subgroups of Lacunary Hyperbolic Groups and Free Products}
\author{Krishnendu Khan}
\date{} 
\maketitle
%\date{} 

\begin{abstract}
A finitely generated group is lacunary hyperbolic if one of its asymptotic cones is an $\mathbb{R}$-tree. In this article we give a necessary and sufficient condition on lacunary hyperbolic groups in order to be stable under free product by giving a dynamical characterization of lacunary hyperbolic groups. Also we studied limits of elementary subgroups as subgroups of lacunary hperbolic groups and characterized them. Given any countable collection of increasing union of elementary groups we show that there exists a lacunary hyperbolic group whose set of all maximal subgroups is the given collection. As a consequence we construct a finitely generated divisible group. First such example was constructed by V. Guba in \cite{Gu86}. In section \color{blue}\ref{ripssection}\color{black}\ we show that given any finitely generated group $Q$ and a non elementary hyperbolic group $H$, there exists a short exact sequence $1\rightarrow N\rightarrow G\rightarrow Q\rightarrow 1$, where $G$ is a lacunary hyperbolic group and $N$ is a non elementary quotient of $H$. Our method allows to recover \cite[Theorem 3]{AS14}.	In section \color{blue}\ref{maximalvonneumann}\color{black}, we extend the class of groups $\mathcal{R}ip_{\mathcal{T}}(Q)$ considered in \cite{CDK19} and hence give more new examples of property $(T)$ von Neumann algebras which have maximal von Neumann subalgebras without property $(T)$. 
\end{abstract}	
\thispagestyle{empty}
\tableofcontents

\section{Introduction}
    The interplay between general group theory and combinatiorial investigation of specific groups having some geometric structure has been exploited by Gromov by introducing and developing the notion of hyperbolic groups in his celebrated article \cite{Gr87}. In this influencing work Gromov proposed wide range research program. \par   
 Von Neumann introduced the notion of amenable groups in \cite{vN29} around 1929 and conjectured that a group $G$ is non amenable if and only if $G$ contains a free subgroup of rank $2$ after showing that no amenable group contains a free subgroup of rank $2$. This conjecture is shown to be false by A. Olshanskii in in \cite{Ol79}. Later in \cite{Ol93}, A. Olshanskii developed small cancellation theory over hyperbolic groups and showed that one can construct non abelian groups without free subgroups by taking quotient of non elementary hyperbolic groups, which can be realized as a direct limit of hyperbolic groups. One can infact get non amenable groups in the limit by starting with property (T) hyperbolic group. \par
 Intuitive idea behind an asymptotic cone of a metric space is to study "large scale" geometry of a metric space by viewing it from "infinite distance". The term "asymptotic cone" was coined by Gromov in \cite{Gr81} to prove that a group of polynomial growth is virtually nilpotent. Later the notion of asymptotic cone was formally developed for arbitrary finitely generated groups by L. van den Dries, A.J. Wilkie in \cite{vdDW84}. Using asymptotic cones, one can characterize several important classes of groups. For example, groups of polynomial growth are precisely groups with all asymptotic cones locally compact [\cite{Gr81,Pau89,D02}].In \cite{Gr87} M. Gromov characterized hyperbolic groups in terms of asymptotic cones.
    
    \begin{theorem}[\cite{Gr87}]
    	A finitely generated group is hyperbolic if and only if all of its asymptotic cones are $\mathbb{R}$-tree.
    \end{theorem} 
  It was shown in \cite{DP99} that asymptotic cones of non-elementary hyperbolic groups are all isometric to the complete homogeneous $\mathbb{R}$-tree of valence continuum. Asymptotic cones of quasi-isometric spaces are bi-Lipschitz equivalent and hence the topology of an asymptotic cone of a finitely generated group does not depend on the choice of the generating set. For a survey of results on asymptotic cones and quasi isometric rigidity results see \cite{D02,DS05}. 
    In the paper \cite{OOS07}, the authors have considered lacunary hyperbolic groups (see definition, sec. \color{blue}\ref{lacunarydef}\color{black}) 
    \color{black}\ and showed that there exists a finitely generated group, one of whose asymptotic cone $\mathcal{C}$ has non trivial countable fundamental group ($\pi(\mathcal{C})\equiv \mathbb{Z}$). We studied those limit groups in this article via hyperbolicity function (see section 2,3) and characterized lacunary hyperblic groups by hyperbolicity functions which measures the thinness of triangles in a geodesic metric spaces. We show that the thinness of triangles can grow sublinearly with respect to the size of perimeter of triangles in hyperbolic metric space and thinness of triangle grows sublinearly with respect to an increasing sequence of the size of perimeter of triangles in lacunary hyperbolic metric space. 
    \begin{theorem}[Theorem \color{blue}\ref{lhgdef}\color{black}]
    		Let G be a finitely generated group with generating set $S$. Then $G$ is lacunary hyperbolic if and only if the hyperbolicity function $f_{G}$ of the corresponding Cayley graph $\Gamma(G,S)$ satisfies $\liminf_{t\rightarrow\infty}\frac{f_{G}(t)}{t}=0$.
    \end{theorem}
   One advantage of having a dynamical characterization is that one can get a necessary and sufficient condition when lacunary hyperbolicity is preserved under free product. In general lacunary hyperbolicity is not stable under free product \cite[Example 3.16]{OOS07}. See Definition \color{blue}\ref{synchrodef}\color{black}\ for the definition of synchronized LH. 
   \begin{theorem}[Theorem \color{blue}\ref{synchro}\color{black}]
   	Let $G_1=G* H$. Then $G_1$ is lacunary hyperbolic if and only if $G$ and $H$ are synchronized LH.
   \end{theorem}  
     Maximal elementary subgroup $E(g)$ of a non elementary hyperbolic group $G$, containing an element $g$ of infinite order, is characterized by $E(g)=\{x\in G| \ x^{-1}g^nx=g^{\pm n} \ for \ some \ n=n(x)\in\mathbb{N}-\{0 \}\}$. $E(g)$ in a hyperbolic group. In section \color{blue}\ref{section4}\color{black}\ of this article we take this algebraic definition (denote by $E^{\mathcal{L}}(g)$) and characterize for lacunary hyperbolic groups. 
    \begin{theorem}[Theorem \color{blue}\ref{elementarygroup}\color{black}]
    	Let $G$ be a lacunary hyperbolic group and $g\in G$ be an infinite order element. Then $E^{\mathcal{L}}(g)$ has a locally finite normal subgroup $N\triangleleft G$ such that:\\
    	Either $E^{\mathcal{L}}(g)/N$ is an abelian group of Rank 1(i.e. $E^{\mathcal{L}}(g)/N<(\mathbb{Q},+,\cdot )$) or $E^{\mathcal{L}}(g)/N$ is an extension of a rank one group by involutive automorphism(i.e, $a\rightarrow a^{-1}$).
    \end{theorem}	
     Conversely, for any rank one abelian group $E$, there exists a lacunary hyperbolic group $G$ with an infinite order element $g\in G$ such that $E=E^{\mathcal{L}}(g)$. More generally we obtained the following,
     \begin{theorem}[Theorem \color{blue}\ref{monster}\color{black}]\label{rankoneintro}
     	For any torsion free non elementary hyperbolic group $G$ and a countable family $\mathscr{F}:=\{Q^i_m\}_{i\in\mathbb{N}}$ of subgroups of $(\mathbb Q,+)$, there exists a non elementary, torsion free, non abelian lacunary hyperbolic quotient $G^{\mathscr{F}}$ of $G$ such that the set of all maximal subgroups of $G^{\mathscr{F}}$ is equal up to isomorphism to $\{Q^i_m\}_{i\in\mathbb{N}}$ i.e, every maximal subgroup of $G^{\mathscr{F}}$ is isomorphic to $Q^i_m$ for some $i\in\mathbb{N}$ and for every $i$ there exists a maximal subgroup of $G^{\mathscr{F}}$ that is isomorphic to $Q^i_m$.
     \end{theorem}  
     Taking $\mathscr{F}=\{\mathbb{Q}\}$ in previous theorem we recover the following theorem by V.S. Guba,
     \begin{corollary}\cite[Theorem 1]{Gu86}
     	There exists a non trivial finitely generated torsion free divisible group. 
     \end{corollary}   
     By taking $\mathscr{F}=\{\mathbb{Z}\}$ we get following,
     \begin{corollary}\cite[Corollary 1]{Ol93}
     	Every non cyclic torsion free hyperbolic group $G$ has a non cyclic quotient $\bar{G}$ such that every proper non trivial subgroups of $\bar{G}$ is infinite cyclic.
     \end{corollary}
   A group $G$ has unique product property whenever for all pairs of non empty finite subsets $A$ and $B$ of $G$ the set of products $AB$ has an element $g\in G$ with a unique representation of the form $g=ab$ with $a\in A$ and $b\in B$. In \cite{RS87}, Rips and Segev gave the first examples of torsion-free groups without the unique product property. Other examples of torsion free groups without unique product can be found in \cite{Pas77,Pro88,Car14,Al91}. As a corollary of Theorem \color{blue}\ref{rankoneintro}\color{black}\ we get following,
   \begin{corollary}[Corollary \color{blue}\ref{monsteruniqueproduct}\color{black}]
   	For every rank one abelian group $Q_m$, there exists a non elementary, torsion free, property $(T)$, lacunary hyperbolic group without the unique product property $G^{Q_m}$ such that any maximal subgroup of $G^{Q_m}$ is isomorphic to $Q_m$.\par 
   	In particular, there exist a non trivial property $(T)$ torsion free divisible lacunary hyperbolic group without the unique product property.
   \end{corollary}
In subsection \color{blue}\ref{rankonebyfinitesection}\color{black}\ we discuss locally finite by rank one abelian subgroups of lacunary hyperbolic groups. We denote the class of increasing union of elementary subgroups as ${}_{rk-1}\mathscr{E}_{F}$ (see Notation \color{blue}\ref{notationelementary}\color{black}). Even though Theorem \color{blue}\ref{rankoneintro}\color{black} \ can be thought of as a corollary of the following theorem, we choose to discuss those in two different subsections in order to make a clear exposition.   
\begin{theorem}[Theorem \color{blue}\ref{rk1byfinite}\color{black}]
	Let $G$ be a torsion free non elementary hyperbolic group and $\mathscr{C}:=\{E^j \}_{j\in\mathbb{N}}$ be a countable collection of groups with $E^j\in {}_{rk-1}\mathscr{E}_{F}$ for all $j\geq 1$. Then there exists a non elementary lacunary hyperbolic quotient $G^{\mathscr{C}}$ of $G$ such that $\{E^{\El}(h)\ | \ h\in (G^{\mathscr{C}})^0 \}=\mathscr{C}$.\par
	Moreover $\mathscr{C}$ is the set of all maximal proper subgroups of the group $G^{\mathscr{C}}$.
\end{theorem}

    We also investigate Rips type construction in this article in order to construct lacunary hyperbolic groups from a given finitely generated group by closely following the proof strategy of E. Rips for hyperbolic group along with small cancellation condition developed by A. Olshanskii in \cite{Ol93}, later in \cite{OOS07} by A.O'lshanskii, D. Osin and M. Sapir, and in \cite{Os'10} by D. Osin.
    \begin{theorem}[Theorem \color{blue}\ref*{Rips}\color{black}]\label{ripsintro}
    		Let $Q$ be a finitely generated group and $H$ be a non elementary hyperbolic group. Then there exist groups $G$ and $N$, for which we get a short exact sequence;
    	$$1\rightarrow N\rightarrow G\rightarrow Q\rightarrow 1$$
    	such that;
    	\begin{enumerate} 
    	\item $G$ is a lacunary hyperbolic group.
    	\item $N$ is a $2$ generated non elementary quotient of $H$.
    	\item If $H$ is torsion free then so are $G$ and $N$.
    	\end{enumerate}    	 
    \end{theorem}
\begin{remark}
	$G$ cannot be hyperbolic if $Q$ has no finite presentation.
\end{remark}
Note that if one takes $H$ to be a torsion free, property $(T)$, hyperbolic group without the unique product property in the Theorem \color{blue}\ref{ripsintro}\color{black}, then one can recover the following theorem (see, Section \color{blue}5\color{black}\ for details),
\begin{corollary}[Corollary \color{blue}\ref{uniqueproduct}\color{black},\ Theorem 3,\cite{AS14}]
	Let $Q$ be a finitely generated group. Then there exists a short exact
	sequence
	\begin{align}
	1\rightarrow N\rightarrow G\rightarrow Q\rightarrow 1\nonumber 
	\end{align}

	such that;
	\begin{itemize}
		\item $G$ is a torsion-free group without the unique product property which is a direct limit of Gromov hyperbolic groups.
		\item $N$ is a subgroup of $G$ with Kazhdan’s Property $(T)$ and without the unique product property.
	\end{itemize}
\end{corollary}   
The group von Neumann algebra $\El G$ corresponding to a countable discrete group $G$ is the weak operator topology closer of the group algebra $\mathbb{C}G$ inside $\mathscr{B}(\ell^2 G)$, the space of bounded linear operators on the Hilbert space $\ell^2G$. The group constructed in Theorem \color{blue}\ref{rankoneintro}\color{black}\ can be used to construct interesting concrete examples of a group such that the corresponding group von Neumann algebra has property $(T)$ while all of its maximal subalgebras does not have property $(T)$. Theorem \color{blue}4.4\color{black}\ in \cite{CDK19} shows how the Belegradek-Osin's group Rips construction techniques and Ol'shanski's type monster groups can be used in conjunction with Galois correspondence results for II$_1$ factors $\grave{a}$ $la$ Choda \cite{Ch78} to produce many maximal von Neumann subalgebras arising from group. In particular, many examples of II${_1}$ factors are constructed in \cite{CDK19} with property (T) that have maximal von Neumann subalgebras without property (T). In section \color{blue}\ref{maximalvonneumann}\color{black}, we extend the class of groups $\mathcal{R}ip_{T}(Q)$ considered in \cite{CDK19} by generalizing \cite[Theorem 3.10]{CDK19} and hence provide new examples of such groups (see section \color{blue}\ref{maximalvonneumann}\color{black}\ for details). For groups with property $(T)$ considered in Theorem \color{blue}\ref{rankoneintro}\color{black}\ there exists a property $(T)$ group $N$ such that $Q\hookrightarrow Out(N)$ and $[Out(N):Q]<\infty$ by \cite{BO06}. Denote $N\rtimes Q$, the groups getting from \cite{BO06} as described above, by the $\mathcal{R}ip_{T}(Q)$ for fixed $Q$. Note that $N\rtimes Q$ has property $(T)$.
\begin{theorem}[Theorem \color{blue}\ref{maximalvN}\color{black}]  $\mathcal M^i_m$ is a maximal von Neumann algebra of $\mathcal M$ for every $i$. In particular, when $N\rtimes Q\in \mathcal Rip_{\mathcal T}(Q)$  then for every $i$, $\mathcal M^i_m$ is a non-property (T) maximal von Neumann subalgebra of a property (T) von Neumann algebra $\mathcal M$.   
\end{theorem}

%%%%%%%%%%%%%%%%%%%%%%%%%%%%%%%%%%%%%%%%%%%%%%%%%%%%%%%%%%%%%%%%%%%%%%%%%%%%%%%%%%%%%%%%%%%%%%%%%%%%%%%%%%%%%%%%%%%%%%%%%%%%%%%%%%%%%%%%%%%%%%%%%%%%%%%%%%%%%%%%%%%%%%%%%%%%%%%%%%%%%%%%%%%%%%%%%%%%%%%%%%%%%%%%%%%%%%%%%%%%%%%%%%%%%%%%%%%%%%%%%%%%%%%%%%%%%%%%%%%%%%%%%%%%%%%%%%%%%%%%%%%%%%%%%%%%%%%%%%%%%%%%%%%%%%%%%    
\section{Hyperbolic spaces}
\subsection{Hyperbolic metric spaces}
I this section we shall mention equivalent definitions for hyperbolic spaces and their relations. In later sections we shall use whichever definition is convenient for the purpose.
\begin{definition}[slim-triangle][Rips]
	Let $\delta\geq 0$. A geodesic triangle in a metric space is said to be $\delta$-slim if each of it's side is contained in $\delta $ neighborhood of the union of other two sides. A geodesic metric space $X$ is said to be $\delta$-hyperbolic if all geodesic triangles in $X$ are $\delta$ hyperbolic. When a space is said to be hyperbolic it means that it is $\delta$-hyperbolic for some $\delta>0$.  
\end{definition}

Given a geodesic triangle (formed by joining three points $x,y,z\in X$ by geodesics, denote by $\triangle(x,y,z)$) in a metric space, the triangle equality tells us that there exists unique non-negative numbers $a,b,c$ such that $d(x,y)=a+b$, $d(x,z)=a+c$ and $d(y,z)=b+c$. Consider the metric tree $T_{\triangle}:=T(a,b,c)$, which consists of $3$ vertices with valency $1$ and $1$ vertices with valency 3, $3$ edges of length $a,b$ and $c$. There exists an isometry from $\{x,y,z\}$ to a subset of vertices of $T_{\triangle}$ (vertices with valency 1). Call these vertices $\{v_x,v_y,v_z\}$. This map $\{x,y,z\}\rightarrow \{v_x,v_y,v_z\}$ extends uniquely to a map $\chi_{\triangle}:\triangle \rightarrow T_{\triangle}$ such that restriction to each side of the triangle is an isometry. Denote $o_{\triangle}$ to be the central vertex of $T_{\triangle}$. 
\begin{definition}[$\delta$-thinness and Insize][Gromov,\cite{Gr87}]
	Let $\triangle$ be a geodesic triangle in a metric space $(X,d)$ and consider the map $\chi_{\triangle}:\triangle \rightarrow T_{\triangle}$ as defined above. Let $\delta\geq 0$. $\triangle$ is said to be $\delta$-thin if $p,q\in (\chi_{\triangle})^{-1}(t)$ implies $d(p,q)\leq \delta$ for all $t\in T_{\triangle}$. The diameter of $\chi_{\triangle}^{-1}(o_{\triangle})$ is denoted by insize $\triangle$. \\
	A geodesic metric space is called $\delta $-hyperbolic if there exists a $\delta\geq 0$ such that all geodesic triangles are $\delta$-thin. When a space is said to be hyperbolic it means that it is $\delta$-hyperbolic for some $\delta$. 
\end{definition}
These two definitions are equivalent but hyperbolicity constant may be different for two definitions. In that case one can take hyperbolicity as the maximum of hyperbolicity constants arising from different definition. For proofs we refer \cite{BH99}, chapter $III$.\\
Next property of hyperbolic space can be deduce from cutting hyperbolic $n$-gon into triangles.
\begin{lemma}\cite{Ol93}\label{hyperbolicngon}
	For $n\geq 3$, any side of a geodesic $n$-gon in a $\delta$-hyperbolic space belongs to the $(1+\log_2(n-1))\delta$ neighborhood of the other $(n-1)$ sides.
\end{lemma}

%%%%%%%%%%%%%%%%%%%%%%%%%%%%%%%%%%%%%%%%%%%%%%%%%%%%%%%%%%%%%%%%%%%%%%%%%%%%%%%%%%%%%%%%%%%%%%%%%%%%%%%%%%%%%%%%%%%%%%%%%%%%%%%%%%%%%%%%%%%%%%%%%%%%%%%%%%%%%%%%%%%%%%%%%%%%%%%%%%%%%%%%%%%%%%%%%%%%%%%%%%%%%%%%%%%%%%%%%%%%%%%%%%%%%%%%%%%%%%%%%%%%%%%%%%%%%%%%%%%%%%%%%%%%%%

\subsection{Ultrafilter and asymptotic cone}
In a metric space a segment is defined to be a subset isometric to an interval of the set of real numbers $\mathbb{R}$. A ray is defined to be a segment isometric to $[0,\infty)$.
\begin{definition}[\cite{Pau89}]
	An $\mathbb{R}$-tree is a complete metric space $T$ such that, for all points $x,y$ in $T$, there is a unique compact arc with endpoints $x$ and $y$, such that this arc is a segment.
\end{definition}
\begin{definition}
	A non principle ultrafilter $\omega$ is a finitely additive measure on all subsets $\mathcal{S}$ of natural numbers $\mathbb{N}$ such that ,
	$\omega(\mathcal{S})\in \{0,1 \}, \omega(\mathbb{N})=1$ and $\omega(\mathcal{S})=0$ for all finite subset of $\mathbb{N}$.	\par 
	For a bounded sequence of numbers $\{x_n \}_{n\in \mathbb{N}}$ the ultralimit $\lim^{\omega}_i$ with respect to $\omega$ is the unique number $a$ such that ;
	$$\omega(\{i\in \mathbb{N}| \ |a-x_i|<\epsilon \})=1, \ \ for \ all \ \epsilon>0$$
	Similarly $\lim^{\omega}_i x_i=\infty$ if;
	$$\omega(\{i\in \mathbb{N}| \ x_i>M \})=1, \ \ for \ all \ M<\infty$$
	For two infinite series of real numbers $\{a_n\}$, $\{b_n\}$ we write $a_n=o_{\omega}(b_n)$ if ${\lim^{\omega}}_i \frac{a_i}{b_i}=0$ and $a_n=\Theta_{\omega}(b_n)$ (respectively $a_n=\mathcal{O}_{\omega}(b_n)$) if $0 < {\lim^{\omega}}_i (\frac{a_i}{b_i})<\infty$ (respectively ${\lim_i^{\omega}} (\frac{a_i}{b_i})<\infty$).\par  
	Let $(X_n,d_n)$ be a metric space for every $n\in \mathbb{N}$. Let $\{e_n \}=:e$ be a sequence of points such that $e_n\in X_n$. Consider the set $\mathcal{F}$ of sequences $g=\{g_n \}$ where $g_n\in X_n$ such that $d_n(e_n,g_n)\leq C(g)$ $(C$ is a constant depending only on $g)$. Define an equivalence relation between elements of $\mathcal{F}$ as follows;
	\begin{center}
		$f=\{f_n\}$ is equivalent to $h=\{h_n\}$ in $\mathcal{F}$ if and only if $\lim^{\omega}_id_i(f_i,h_i)=0$
	\end{center}
	The equivalence class of an element $\{g_n\}$ of $\mathcal{F}$ is denoted by $\{g_n\}^{\omega}$. The $\omega$-limit $\lim^{\omega}(X_n)_e$ is the quotient space of $\mathcal{F}$ by this equivalence relation and distance between two elements is defined as follows;
	\begin{center}
		$Dist(f,g)=\lim^{\omega}d_i(f_i,g_i)$ for $f=\{f_i\}^{\omega}$, $g=\{g_i\}^{\omega}$ in $\lim^{\omega}(X_n,d_n)_e$.
	\end{center}
	An asymptotic cone $Con^{\omega}(X,e,l)$ for a metric space $(X,D)$, is an $\omega$-limit of metric space $X_n$ as above, where $X_n=X$ for all $n\in \mathbb{N}$, $d_n=\frac{D}{l_n}$ for a given non decreasing infinite sequence $l=\{l_n\}$ of positive real numbers and for a given sequence of points $e=\{e_n\}$ in $X$.\par 
	If $\{Y_n\}$ is a sequence of subsets of $X$ endowed with the induced metric, define $\lim^{\omega}(Y_n)_e$ to be the subset of $Con^{\omega}(X,e,l)$ consisting of $x\in Con^{\omega}(X,e,l)$ that can be represented by a sequence $\{x_n\}$, where $x_n\in Y_n$.
	\begin{remark}
		The asymptotic cone is a complete metric space. Moreover asymptotic cone $Con^{\omega}(X,e,l)$ is a geodesic metric space if $X$ is a geodesic metric space \cite{DS05,Gr96}.
	\end{remark}
	The asymptotic cone of a finitely generated group $G$ with a word metric is asymptotic cone of it's Cayley graph. Asymptotic cone of a finitely generated group with respect to two generating sets are bi-Lipscitz. \\
	A geodesic path $p$ in $Con^{\omega}(X,e,l)$ is called limit geodesic if $p=\lim^{\omega}(p_n([0,\lambda_n]))_e$ where $p_n:[0,\lambda_n]\rightarrow X$ is geodesic for all $n\in \mathbb{N}$.
	
\end{definition}
\begin{definition}\cite{DS05}
	Let $X$ be a complete geodesic metric space. Let $\mathcal{P}$ be a collection of closed geodesic non-empty subsets (called $''pieces''$). If the following properties are satisfied:\\
	$(T_1)$ Every two different pieces have at-most one common point.\\
	$(T_2)$ Every nontrivial simple geodesic triangle (A simple loop consists with 3 geodesics) in $X$ is contained in one piece.\\
	Then we say $X$ is tree-graded space with respect to $\mathcal{P}$.
\end{definition}
$(T_2')$ Every non trivial simple loop in $X$ is contained in one piece.\\
Here one can replace condition ($T_2$) by ($T_2'$).\\
Here by loop we mean a path with same initial point and terminal point. 
\begin{definition}
	A loop $\gamma:[a,b]\rightarrow X$ is simple if $\gamma(t_1)\neq \gamma(t_2)$ for $t_1\neq t_2$ in $[a,b)$.
\end{definition}
\begin{lemma} [corollary 4.18, \cite{DS05b}]\label{limitgeodesictrg}
	Assume that in an asymptotic cone $Con^\omega(X,e,d)$, a collection of closed subsets $\mathcal{P}$ satisfying ($T_1$) and every non trivial simple geodesic triangle in $Con^\omega(X,e,d)$ whose sides are limit geodesic is contained in a subset from the collection $\mathcal{P}$. Then $\mathcal{P}$ satisfy ($T_2$) i.e, $Con^\omega(X,e,d)$ is tree-graded with respect to $\mathcal{P}$.	
\end{lemma}

\begin{theorem}[\cite{Gr87}]\label{asymptoticgromov}
	A geodesic metric space is hyperbolic if and only if all of its asymptotic cones are $\mathbb{R}$-tree.
\end{theorem}

%%%%%%%%%%%%%%%%%%%%%%%%%%%%%%%%%%%%%%%%%%%%%%%%%%%%%%%%%%%%%%%%%%%%%%%%%%%%%%%%%%%%%%%%%%%%%%%%%%%%%%%%%%%%%%%%%%%%%%%%%%%%%%%%%%%%%%%%%%%%%%%%%%%%%%%%%%%%%%%%%%%%%%%%%%%%%%%%%%%%%%%%%%%%%%%%%%%%%%%%%%%%%%%%%%%%%%%%%%%%%%%%%%%%%%%%%%%%%%%%%%%%%%%%%%%%%%%%%%%%%%%%%%%%%%%%%%%%%%%%%

\subsection{Hyperbolic group and hyperbolicity function}     

\begin{definition}
	Let $(X,d)$ be a geodesic metric space. For every positive real number $t$, define the following set:
	$$S_t:=\{ set \ of \ all \ geodesic \ triangles \ with \ perimeter \ \leq t \}$$
	For a geodesic triangle triangle $\triangle$ with sides $A,B,C$ in $X$ define, 
	$$\delta_{\triangle}:= \underset{p\in A\cup B\cup C}{\sup}d(p,union \ of \ two\ sides\ of\ \triangle \ not \ containing \ p)$$
	Define the hyperbolicity function $f_X:[0,\infty)\rightarrow [0,\infty)$ by $f_X(t):=\underset{\triangle \in S_t}{\sup}\delta_{\triangle}$.
\end{definition}
One can reformulate lemma $\color{blue}{\ref*{hyperbolicngon}}$ for a geodesic metric space in terms of it's hyperbolicity function as follows,
\begin{lemma}\label{hyperbolicgeodesicngon1}
	For $n\geq 3$, any side of a geodesic $n$-gon $T_n$, in a geodesic metric space $X$, belongs to the closed $(1+\log_2(n-1))f_X(|T_n|)$ neighborhood of other $(n-1)$ sides, where $f_X$ is the hyperbolicity function of the geodesic metric space $X$ and $|T_n|$ is the perimeter of the $n$-gon $T_n$.  
\end{lemma}
\begin{definition}
	A function $f:[0,\infty)\rightarrow [0,\infty)$ is called sub-linear if $\lim_{t\rightarrow \infty }\frac{f(t)}{t}=0$.
\end{definition}
\begin{theorem}\label{hyptriangle}
	A geodesic metric space $(X,d)$ is hyperbolic if and only if the hyperbolicity function $f_X$ is sub-linear. 
\end{theorem}

\begin{proof}
	If $X$ is $\delta$ hyperbolic space for some $\delta>0$, then by the definition of hyperbolicity function $f_X$ is bounded by $\delta$ and hence sub-linear.\par 
	Suppose the hyperbolicity function $f_X$ of $(X,d)$ is sub-linear. Lets assume that $(X,d)$ is not hyperbolic. By ${\color{blue}\ref*{asymptoticgromov}}$, there exists an asymptotic cone which is not  $\mathbb{R}$-tree (Say, $Con^{\omega}(X,e,d=\{d_n\})$). Hence by definition of $\mathbb{R}$-tree, there exists a non-trivial simple loop in $Con^{\omega}(X,e,d)$. As in the definition of tree graded space ($T_1$) and ($T_2'$) are equivalent, we can assume that there exists a non-trivial simple geodesic triangle in $Con^{\omega}(X,e,d)$. According to lemma $\color{blue}{\ref*{limitgeodesictrg}}$ applied to the collection $\mathcal{P}$ of all one-element subsets of $Con^{\omega}(X,e,d)$, to show that $Con^{\omega}(X,e,d)$ is a tree it suffices
	\begin{figure} 
		\centering
		\begin{tikzpicture}
		\draw (1,0) -- (5,0);
		\draw [dashed] (7,2) parabola (5,0);
		\draw [dashed] (-1,2) parabola (1,0);
		\draw (7,2) parabola (4,5);
		\draw [dashed] (1,4) parabola (4,5);
		\draw (-1,2) parabola (1,4);
		%\draw [blue] (-1,2) parabola (5,0);
		%\draw [blue] (-1,2) parabola (4,5);
		%\draw [blue] (5,0) parabola (4,5);
		\node[left, red] at (0.15,0.75){$P^1_n=o_{\omega}(d_n)$};
		\node[right, red] at (5.85,0.75){$P^2_n=o_{\omega}(d_n)$};
		\node[left, red] at (3,4.75){$P^3_n=o_{\omega}(d_n)$};
		\draw [red] (4.90,3.5) parabola (3,0);
		\node[left] at (4.25,0.5){$\leq \alpha_n$};
		\node[left] at (0,5){$\mathbf{H_n}$};
		%\node[blue] at (3,3){$\mathbf{\triangle_n}$};
		\node[below] at (1,0){$a_n^{\prime}$};
		\node[below] at (5,0){$b_n$};
		\node[right] at (7,2){$b_n^{\prime}$};
		\node[left] at (-1,2){$a_n$};
		\node[above] at (4,5){$c_n$};
		\node[above] at (1,4){$c_n^{\prime}$};
		\node [below, blue] at (3,-0.25){${A_n}$};
		\node [right, blue] at (5,4){${B_n}$};
		\node [left, blue] at (0,3){${C_n}$};
		\end{tikzpicture}
		\caption{$H_n$} \label{1}
	\end{figure}
	to prove that it contains no simple non-trivial limit geodesic triangles\footnote{by limit geodesic triangles we mean a triangle in $Con^{\omega}(X,e,d)$ whose sides are limit geodesics}.\par 
	Let $\triangle_{\infty}$ be a simple non-trivial limit geodesic triangle in $Con^{\omega}(X,e,d)$ with sides A, B and C with,
	\begin{center} 
		$\lim^{\omega}{A_n}=A, 
		\lim^{\omega}{B_n}=B, 
		\lim^{\omega}{C_n}=C$
	\end{center}
	where $A_n,B_n,C_n$ are geodesics in $X$ with endpoints $a_n',b_n$ for $A_n$ and $b_n',c_n$ for $B_n$ and $c_n',a_n$ for $C_n$, for all $n\in \mathbb{N}$.
	Let ${\triangle}_{\infty}$ is approached by hexagons $H_n$, formed by vertices $a_n{a_n}'b_n{b_n}'c_n{c_n}'$ in $X$ with $\ell (P^i_n)$\footnote{$\ell(P^i_n)$ denotes the length of the geodesic $P^i_n$ for $i=1,2,3$ and for all $n\in \mathbb{N}$}$=o_{\omega}(d_n)$ for $i=1,2,3$, where, $ P^1_n$ is a geodesic joining $a_n,a_n'$ in $X$ and respectively $P^2_n$ is a geodesic joining $b_n,b_n'$ in $X$ and $P^3_n$ is a geodesic joining $c_n,c_n'$ in $X$. \par 
	Denote the perimeter of the hexagon $H_n$ by $|H_n|$. Let $A$ be a nontrivial side of $\triangle_{\infty}$. By lemma $\color{blue}{\ref*{hyperbolicgeodesicngon1}}$, $A_n$ belongs to closed $4f_X(|H_n|)$ neighborhood of other $5$ sides in $H_n$. In particular $A_n$ is contained in closed $\alpha_n$ neighborhood of $B_n\cup C_n $, where \begin{align}\label{A} 
	\alpha_n=4f_X(|H_n|) + \max_{i=1,2,3}\{\ell(P^i_n)  \}.
	\end{align}
	We have,
	\begin{align}\label{B}
	\ell(P^i_n)=o_{\omega}(d_n)\\
	\ell(A_n) = \Theta_{\omega}(d_n)\nonumber
	\end{align} 
	Hence we get $|H_n|=\Theta_{\omega}(d_n)$. As the function $f_X$ is sub-linear, we also have, \begin{align}\label{C} 
	f_X(|H_n|)=o_{\omega}(|H_n|)=o_{\omega}(d_n).
	\end{align}
	Finally we get that $A$ is contained in the union of other two sides in $Con^{\omega}(X,e,d)$, as $\alpha_n=o_{\omega}(d_n)$ (combining $\color{blue}{(\ref*{A}),(\ref*{B}),(\ref*{C})}$). This contradicts out assumption that the triangle $\triangle_{\infty}$ is simple. Hence $Con^{\omega}(X,e,d)$ is an $\mathbb{R}$-tree. That contradicts our assumption that $X$ is not hyperbolic. 
\end{proof}

%%%%%%%%%%%%%%%%%%%%%%%%%%%%%%%%%%%%%%%%%%%%%%%%%%%%%%%%%%%%%%%%%%%%%%%%%%%%%%%%%%%%%%%%%%%%%%%%%%%%%%%%%%%%%%%%%%%%%%%%%%%%%%%%%%%%%%%%%%%%%%%%%%%%%%%%%%%%%%%%%%%%%%%%%%%%%%%%%%%%%%%%%%%%%%%%%%%%%%%%%%%%%%%%%%%%%%%%%%%%%%%%%%%%%%%%%%%%%%%%%%%%%%%%%%%%%%%%%%%%%%%%%%%%%%%%%%%%%%%%%

\section{Lacunary hyperbolic groups}\label{lacunarydef}

\begin{definition}[\cite{OOS07}]
	A metric space $X$ is called lacunary hyperbolic if one of the asymptotic cones of $X$ is an $\mathbb{R}$-tree. A finitely generated group is called lacunary hyperbolic if it has a finite generating set and the corresponding Cayley graph is lacunary hyperbolic.
\end{definition}
\begin{remark}
	Note that lacunary hyperbolicity does not depend on finite generating set. It follows from Theorem \color{blue}\ref{OOS}\color{black}\ below.
\end{remark}

\begin{definition}
	Let $\alpha: G\rightarrow H$ be a group homomorphism and $G=\langle A\rangle ,H=\langle B \rangle$. The injectivity radius $r_A(\alpha)$ is the radius of largest ball centered at identity of $G$ in the Cayley graph of $G$ with respect to $A$ on which the restriction of $\alpha$ is injective.
\end{definition} 
\begin{remark}
	In the above setting, $r_A(\alpha)$ can be $\infty$, for example when $\alpha$ is injective.
\end{remark}

\begin{theorem}\cite[Theorem 1.1~]{OOS07}\label{OOS}
	Suppose $G$ be a finitely generated group. Then following are equivalent:\\
	$\mathbf{a.)}$ $G$ is lacunary hyperbolic group.\\
	$\mathbf{b.)}$ There exists a scaling sequence $d=(d_n)$, such that $Con^{\omega}(G,d)$ is an $\mathbb{R}$-tree for all non-principal ultrafilter $\omega$.\\
	$\mathbf{c.)}$ $G$ is the direct limit of a sequence of finitely generated hyperboolic groups and epimorphisms;\\
	\begin{align}\label{lacunarylim} 
	G_1\overset{\alpha_1}{\longrightarrow} G_{2}\overset{\alpha_{2}}{\longrightarrow}\cdots G_i\overset{\alpha_i}{\longrightarrow} G_{i+1}\overset{\alpha_{i+1}}{\longrightarrow}G_{i+2}\overset{\alpha_{i+2}}{\longrightarrow}\cdots 
	\end{align}
	where $G_i$ is generated by a finite set $\langle S_i\rangle $ and $\alpha_i(S_i)=S_{i+1}$. Also $G_i$'s are $\delta_i$ hyperbolic where $\delta_i$=$o(r_{S_i}(\alpha_i))$ (where $r_{S_i}(\alpha_i)$=injective radius of $\alpha_i$ w.r.t. $S_i$ ).
\end{theorem}

The following theorem gives us a characterization of lacunary hyperbolic group in terms of the hyperbolicity function of the corresponding Cayley graph. 

\begin{theorem}\label{lhgdef}
	Let G be a finitely generated group with generating set $S$. Then $G$ is lacunary hyperbolic if and only if the hyperbolicity function $f_{G}$ of the corresponding Cayley graph $\Gamma(G,S)$ satisfies $\liminf_{t\rightarrow\infty}\frac{f_{G}(t)}{t}=0$.
\end{theorem}

\begin{proof}
	Suppose $G$ be lacunary hyperbolic group. 
	Consider any geodesic triangle of perimeter less than or equal to $r_{S_i}(\alpha_i)$, then we use part $c.)$ of  Theorem~${\color{blue}\ref*{OOS}}$ and get that $f_G(r_{S_i}(\alpha_i))=\mathcal{O}(\delta_i)=o(r_{S_i}(\alpha_i))$, as $\delta_i=o(r_{S_i}(\alpha_i))$. Hence,
	$$\lim_{n\rightarrow\infty}\frac{f_G(r_{S_i}(\alpha_i))}{r_{S_i}(\alpha_i)}=0\Rightarrow \liminf_{n\rightarrow\infty}\frac{f_G(t)}{t}=0.$$
	\\
	\\
	Suppose the hyperbolicity function $f_G$ satisfies the following condition: $$\liminf_{t\rightarrow\infty}\frac{f_G(t)}{t}=0$$ Hence there exists a non-decreasing scaling sequence, say $k=(k_n)$ such that:
	\begin{align}\label{S} 
	\lim_{n\rightarrow\infty}\frac{f_G(d_n)}{d_n}=0.
	\end{align} 
	Fix the sequence $d=\{d_n\}$. Then we have,
	\begin{align}\label{lac}
	f_G(d_n)=o(d_n).
	\end{align}
	Let $\omega$ be an arbitrary non principle ultrafilter. According to lemma $\color{blue}{\ref*{limitgeodesictrg}}$ applied to the collection $\mathcal{P}$ of all one-element subsets of $Con^{\omega}(G,d)$ to show that $Con^{\omega}(G,d)$ is an $\mathbb{R}$-tree it is sufficient to prove that it contains no non-trivial simple limit geodesic triangles.
	\\
	Let $\triangle_{\infty}$ be a simple non-trivial limit geodesic triangle in $Con^{\omega}(G,d)$ with sides A, B and C with,
	\begin{center} 
		$\lim^{\omega}{A_n}=A, 
		\lim^{\omega}{B_n}=B, 
		\lim^{\omega}{C_n}=C$
	\end{center}
	where $A_n,B_n,C_n$ are geodesics in $X:=\Gamma(G,S)$ with endpoints $a_n',b_n$ for $A_n$ and $b_n',c_n$ for $B_n$ and $c_n',a_n$ for $C_n$, for all $n\in \mathbb{N}$.
	Let ${\triangle}_{\infty}$ is approached by hexagons $H_n$, formed by vertices $a_n{a_n}'b_n{b_n}'c_n{c_n}'$ in $X$ with $\ell (P^i_n)$\footnote{$\ell(P^i_n)$ denotes the length of the geodesic $P^i_n$ for $i=1,2,3$ and for all $n\in \mathbb{N}$}$=o_{\omega}(d_n)$ for $i=1,2,3$, where, $ P^1_n$ is a geodesic joining $a_n,a_n'$ in $X$ and respectively $P^2_n$ is a geodesic joining $b_n,b_n'$ in $X$ and $P^3_n$ is a geodesic joining $c_n,c_n'$ in $X$. \\
	Denote the perimeter of the hexagon $H_n$ by $|H_n|$. Let $A$ be a nontrivial side of $\triangle_{\infty}$. By lemma $\color{blue}{\ref*{hyperbolicgeodesicngon1}}$, $A_n$ belongs to closed $4f_G(|H_n|)$ neighborhood of other $5$ sides in $H_n$. In particular $A_n$ is contained in closed $\beta_n$ neighborhood of $B_n\cup C_n $, where \begin{align}\label{A1} 
	\beta_n=4f_G(|H_n|) + \max_{i=1,2,3}\{\ell(P^i_n)  \}.
	\end{align}
	We have,
	\begin{align}\label{B1}
	\ell(P^i_n)=o_{\omega}(d_n)\\
	\ell(A_n) = \Theta_{\omega}(d_n)\nonumber\\
	\ell(B_n),\ell(C_n)=\mathcal{O}_{\omega}(d_n)\nonumber 
	\end{align} 
	Hence we get $|H_n|=\Theta_{\omega}(d_n)$. By $\color{blue}{\ref*{S}}$, we also have, 
	\begin{align}\label{C1} 
	f_G(|H_n|)=o_{\omega}(|H_n|)=o_{\omega}(d_n).
	\end{align}
	Finally we get that $A$ is contained in the union of other two sides in $Con^{\omega}(X,e,d)$, as $\beta_n=o_{\omega}(d_n)$ (combining $\color{blue}{(\ref*{A1}),(\ref*{B1}),(\ref*{C1})}$). This contradicts out assumption that the triangle $\triangle_{\infty}$ is simple. Hence $Con^{\omega}(X,e,d)$ is an $\mathbb{R}$-tree. Which implies that the group $G$ is lacunary hyperbolic.
\end{proof}

\begin{remark}
	One can use similar proof to characterize lacunary hyperbolic spaces i.e, a geodesic metric space $(X,d)$ is lacunary hyperbolic if and only if the corresponding hyperbolicity function $f_X$ satisfies $\liminf_{t\rightarrow \infty}\frac{f_X(t)}{t}=0$.
\end{remark}
\begin{definition}
	Let $(X,d_X),(Y,d_Y)$ be metric spaces and $f:X\rightarrow Y$ be a map. The map $f$ is called a $(\lambda,c)$ quasi isometric embedding if for all $x,y\in X$,
	\begin{center}
		$\frac{d_X(x,y)}{\lambda}-c < d_Y(f(x),f(y))< \lambda d_X(x,y)+c$
	\end{center}
	A $(\lambda,c)$ quasi isometric embedding is a $(\lambda,c)$ quasi-isometry if there exists a constant $D\geq0 $ such that for all $y\in Y$ there exists $x\in X$ with $d_Y(f(x),Y)<D$. \par 
	We say $X$ is quasi isometric to $Y$ is there exists a quasi isometry from $X$ to $Y$ for some $\lambda$ and $c$.
\end{definition}
\begin{remark}\cite{BH99}
	Quasi isometry is an equivalence relation.
\end{remark}
\begin{definition}\label{quasigeodesic}
	A $(\lambda,c)$quasi-geodesic in a metric space $(X,d_X)$ is a $(\lambda,c)$ quasi-isometric embedding of $\mathbb{R}$ into $(X,d_X)$. More precisely a map $q:\mathbb{R}\rightarrow (X,d_X)$ such that there exists $\lambda,c>0$ so that for all $s,t\in \mathbb{R}$,
	$$\frac{|s-t|}{\lambda}-c < d_X(f(x),f(y))< \lambda|s-t|+c$$
	is called a $(\lambda,c)$ quasi geodesic.
\end{definition}

\begin{theorem}\cite[Theorem 3.18~]{OOS07}\label{aa}
	Let $G$ be a lacunary hyperbolic group. Then, every finitely generated undistorted (i.e, quasi-isometrically embedded) subgroup of $G$ is lacunary hyperbolic.
\end{theorem}

So it is natural to expect the ``same" result to be true for lacunary hyperbolic space.

\begin{theorem} \label{qisubspace}
	Let $X$ be a lacunary hyperbolic space and $q: Y\rightarrow X$ is a quasi isometric embedding. Then $Y$ is also lacunary hyperbolic space.
\end{theorem}

Proof follows from similar argument from Theorem ({\color{blue}\ref*{aa}}).

\begin{corollary}
	Suppose $X,Y$ are two quasi-isometric geodesic metric space. Then $X$ is lacunary hyperbolic if and only if $Y$ is lacunary hyperbolic.
\end{corollary}

\begin{definition}\label{synchrodef}
	Let $G_i$ be LHG (lacunary hyperbolic group) for $i=1,2,\cdots ,k$ with corresponding hyperbolicity function $f_{G_i}$ for $G_i$ for all $i=1,2,\cdots,k$. Then we call $G_i$'s are synchronized LH if there exists an increasing infinite sequence $\{x_j\}_{j\in \mathbb{N}}$ of real numbers and finite generating sets $X_i$ of $G_i$ for all $i$ such that the hyperbolicity function with respect to generating sets $X_i$, $f_{G_i}$ satisfies $\lim_{j\rightarrow\infty}\frac{f_{G_i}(x_j)}{x_j}=0$ for all $i=1,2,\cdots,k$ (equivalently $\liminf_{t\rightarrow \infty} \frac{\sum_{i=1}^{k}f_{G_i}(t)}{t}=0$).
\end{definition}
Free product of two lacunary hyperbolic group is not necessarily a lacunary hyperbolic (Example 3.16, $\cite{OOS07}$).\par 
\begin{theorem}\label{synchro}
	Let $G=G_0* G_1$ with $G_i$'s LHG. Then $G$ is LHG if and only if $G_0$ and $G_1$ are synchronized LH.
\end{theorem}
\begin{proof}    
Let $G_1,G_2$ are synchronized LH. Let $\Delta=T_1T_2T_3$ be a geodesic triangle in $G:=G_1*G_2$ with geodesic sides $T_1,T_2,T_3$, where $G_i$s are hyperbolic groups with finite presentation $G_i=\la X_i|R_i\ra$ for $i=1,2$. Since $\Delta=1$ in $G$, consider the normal form $\Delta=\prod_{i=1}^{k}g_i^{-1}r_i^{\pm 1}g_i$ in the free group $\mathcal{F}(X_1\cup X_2)$ for some positive integer $k$ where $g_i\in (X_1\cup X_2)^{\pm 1}$ and $r_i\in R_1\cup R_2$. We are going to induct on the sum of lengths of the normal form of $\Delta$. Suppose the sum of length of normal form is $\leq 3$. Then the triangle $\Delta$ belongs to one or the free factors either $G_1$ or $G_2$ since the product is $1$. Let us assume that the sum of lengths of normal form is $n\geq 4$. Let we have a cancellation in $T_1T_2=T_1'xyT_2'$ where $xy=1$ in $G_1$ (respectively in $G_2$). Then the triangle $\Delta $ can be written as union of triangle $\Delta_1=xy$ (which is a bigon with non trivial side $x$ and $y$) and the triangle $\Delta_2$ with sides $T_1',T_2'$ and $T_3$. Notice that sum of lengths normal form of $\Delta_1,\Delta_2$ say $n_1,n_2$ respectively, are strictly less than $n$. Hence we get
\begin{equation} 
 f_{G}(n=n_1+n_2)   \leq \max\{f_{G_1}(n_1),f_{G_1}(n_2), f_{G_2}(n_1), f_{G_2}(n_2)\}
   \leq \max\{ f_{G_1}(n),f_{G_2}(n)\}.\nonumber 
 \end{equation} 
Last inequality is coming from the fact that hyperbolicity function is non  decreasing. When there is a cancellation $xyz=1$ in $G_1$ (or in $G_2$) then we are in the following setting, $T_1=T_1'x, T_2=y$ and $T_3=zT_2'$, where $xyz=1$ in $G_1$ (or in $G_2$). In this case we can also write the triangle $\Delta$ as union of triangle $\Delta$ formed by $x,y,z$ and the triangle $\Delta_2$ (bigon) formed by $T_1',T_2'$ with lengths of normal forms $n_1,n_2$. Then we get 
\begin{equation} 
f_{G}(n=n_1+n_2)   \leq \max\{ f_{G_1}(n_1),f_{G_1}(n_2), f_{G_2}(n_1), f_{G_2}(n_2)\}
\leq \max\{ f_{G_1}(n),f_{G_2}(n)\}.\nonumber 
\end{equation} 
Hence by induction we can see that $f_G(t)\leq \max\{ f_{G_1}(t),f_{G_2}(t)\}$ for all $t$. Which implies that,
\begin{align}
\liminf_{t\rightarrow \infty}\frac{f_G(t)}{t}\leq \liminf_{t\rightarrow\infty}\frac{f_{G_1}(t)+f_{G_2}(t)}{t} \rightarrow 0\nonumber. 
\end{align}        
Hence $G$ is a lacunary hyperbolic group.\par     
    
	For the other direction let $G=G_1* G_2$ is LHG. Note that we have 
	\begin{align} 
		f_{G_i}(t)\leq f_G(t)\ \textit{for all } t> 0\ \textit{and for all }i=1,2.
	\end{align} 
	since every geodesic triangles in $G_i$ is also a geodesic triangle in $G$ for all $i=1,2$ as the embedding $G_i\rightarrow G_1*G_2$ is a $(1,0)$ quasi isometric embedding for all $i=1,2$. There exists a sequence, say $\{ y_i\}$ such that, $\lim_{i\rightarrow\infty }\frac{f_G(y_i)}{y_i}=0$. So we have $\lim_{i\rightarrow\infty }\frac{f_{G_j}(y_i)}{y_i}=0$ for all $j=1,2$. Hence by definition $G_1,G_2$ are synchronized LH. 
\end{proof}	

\begin{remark}\label{freeprodlacuremark}
	Above proposition is in fact true for free product of finite number of LHGs i.e, for a finite number of lacunary hyperbolic groups $\{G_i\}_{i=1}^{k}$, $G=*^{k}_{i=1}G_i$ is LHG if and only if $\{G_i\}$ are synchronized LH. In particular $*_{i=1}^kG$ is a lacunary hyperbolic group whenever $G$ is a lacunary hyperbolic group for any positive integer $k$.
\end{remark}

%%%%%%%%%%%%%%%%%%%%%%%%%%%%%%%%%%%%%%%%%%%%%%%%%%%%%%%%%%%%%%%%%%%%%%%%%%%%%%%%%%%%%%%%%%%%%%%%%%%%%%%%%%%%%%%%%%%%%%%%%%%%%%%%%%%%%%%%%%%%%%%%%%%%%%%%%%%%%%%%%%%%%%%%%%%%%%%%%%%%%%%%%%%%%%%%%%%%%%%%%%%%%%%%%%%%%%%%%%%%%%%%%%%%%%%%%%%%%%%%%%%%%%%%%%%%%%%

\section{Algebraic properties of lacunary hyperbolic groups}\label{section4}
\subsection{Small cancellation theory}
\subsubsection{van Kampen diagrams}
Given a word $W$ in alphabets $\mathcal{S}$, we denote its length by $\|W\|$. We also write $W\equiv V$ to express the letter-for-letter equality for words $U,V$. \par 
Let $G$ be a group generated by alphabets $\mathcal{S}$. A van Kampen diagram $\triangle$ over a presentation 
\begin{align} \label{vankampen}
G=\langle S|\mathcal{R}\rangle
\end{align} 
where $\mathcal{R}$ is cyclically reduced words over alphabets $S$, is a finite, oriented, connected, planar 2-complex endowed with a labeling function $Lab:E(\triangle)\rightarrow \mathcal{S}^{\pm 1}$, where $E(\triangle)$ denotes the set of oriented edges of $\triangle$, such that $Lab(e^{-1})\equiv (Lab(e))^{-1}$. Given a cell $\Pi$ of $\triangle$, $\partial \Pi$ denotes its boundary; similarly $\partial \triangle$ denote boundary of $\triangle$. The labels of $\partial \triangle$ and $\partial \Pi$ are defined up to cyclic permutations. Also one additional requirment is that the label for any cell $\Pi$ of $\triangle$ is equal to (up to a cyclic permutation) $R^{\pm 1}$, where $R\in\ \mathcal{R}$.\par 
By van Kampen lemma, a word $W$ over alphabets $\mathcal{S}$ represents the identity element in the group given by (\color{blue}{\ref{vankampen}}\color{black}) if and only if there exists a van Kampen diagram $\triangle$ over (\color{blue}{\ref{vankampen}}\color{black}) such that, $Lab(\partial\triangle)\equiv W$. [Ch.5, Theorem 1.1]\cite{LS77}.

%%%%%%%%%%%%%%%%%%%%%%%%%%%%%%%%%%%%%%%%%%%%%%%%%%%%%%%%%%%%%%%%%%%%%%%%%%%%%%%%%%%%%%%%%%%%%%%%%%%%%%%%%%%%%%%%%%%%%%%%%%%%%%%%%%%%%%%%%%%%%%%%%%%%%%%%%%%%%%%%%%%%

\subsubsection{Small cancellation over hyperbolic groups}
Let $G=\langle X\rangle$ be a finitely generated group. 
\begin{definition}
	A word $W$ in the alphabet $X^{\pm 1}$ is called $(\lambda,c)$-quasi geodesic (respectively geodesic) in $G$ if any path in the Cayley graph $\Gamma(G,X)$ labeled by $W$ is $(\lambda,c)$-quasi geodesic (respectively geodesic).
\end{definition}
Let $G=\langle X\rangle$ be a finitely generated group, and let $\mathcal{R}$ be a symmetric set of words (i.e. it is closed under operations of taking cyclic shifts and inverse of words and all the words are cyclically reduces) from $X^*$. %we write $U\equiv V$ for two words $U$ and $V$ in some alphabet to express letter-by-letter equality.
\begin{definition}\cite{LS77}
	Let \begin{align} 
	G=\langle X\rangle 
	\end{align} 
	be a group generated by $X$, and $\mathcal{R}$ be a symmetrized set of reduced words in a finite set of alphabets $X^{\pm 1}$. A common initial sub-word of any two distinct words in $\mathcal{R}$ is called a piece. We say that $\mathcal{R}$ satisfies $C'(\mu)$ condition if any piece contained (as a sub-word) in a word $R\in\mathcal{R}$ has length less than $\mu\|R\|$.  
\end{definition}

Let $G$ be a group generated by a set $X$. A subword $U$ of a word $R\in \mathcal{R}$ is called an $\epsilon$-piece (reference \cite{Ol93}) for $\epsilon\geq 0$ if there exists a word $R'\in \mathcal{R}$ such that
\begin{enumerate}[(a)]
	\item $R\equiv UV$ and $R^{\prime}\equiv U^{\prime}V^{\prime}$ for some $U^{\prime},V^{\prime}\in \mathcal{R}$.
	\item $U^{\prime}= YUZ$ in $G$ for some $Y,Z\in X^*$ where $\| Y\|,\| Z\| \leq \epsilon$.
	\item $YRY^{-1}{\neq} R^{\prime}$ in the group $G$.
\end{enumerate}

 %$U^{\prime}= YUZ$ in $G$ for some $Y,Z\in X^*$ where $\| Y\|,\| Z\| \leq \epsilon$. \\
%$(3)$ $YRY^{-1}{\neq} R^{\prime}$ in the group $G$.\\ \\
It is said that the system satisfies the $C(\lambda,c,\epsilon,\mu,\rho)$-condition for some $\lambda\geq 1,c\geq 0,\epsilon\geq0,\mu>0,\rho>0$ if,
\begin{enumerate}[a)] 
\item $\| R\| \geq \rho$ for any $R\in \mathcal{R}$.
\item any word $R\in \mathcal{R}$ is a $(\lambda,c)$-quasi geodesic.
\item  for any $\epsilon$-$piece$ of any word $R\in \mathcal{R}$, the inequalities $\| U\|,\| U^{\prime } \| <\mu \| R\|$ holds.
\end{enumerate}
Let $U^{\pm 1}$ be a subword of $R\in \mathcal{R}$ and we have,
\begin{enumerate} 
\item $R\equiv UVU^{\prime} V^{\prime}$ for some $V,U^{\prime} ,V^{\prime}\in X^*$,
\item $U^{\prime}=YU^{\pm 1}Z$ in the group $G$ for some words $Y,Z\in X^*$ where $\| Y\| ,\| Z\| \leq \epsilon$.
\end{enumerate}
Then we call $U$ an $\epsilon^{\prime}$-$piece$ of the word $R$. If $\mathcal{R}$ satisfies the $C(\lambda,c,\epsilon,\mu,\rho)$-condition and in addition for all $R\in \mathcal{R}$, the above decomposition implies $\| U\| , \| U^{\prime} \| < \mu \| R\|$, we say that $\mathcal{R}$ satisfies $C'(\lambda,c,\epsilon,\mu,\rho)$-condition. 
\begin{definition}[Definition 4.3, \cite{OOS07}]
	Let $\epsilon\geq 0,\mu \in (0,1),$ and $\rho>0$. We say that a symmetrized set $\mathcal{R}$ of words over the alphabet $X^{\pm 1}$ satisfies condition $C(\epsilon,\mu,\rho)$ for the group $G$, if
	\begin{enumerate}[label=(\subscript{C}{{\arabic*}})]
	\item All words from $\mathcal{R}$ are geodesic in $G$.

	\item $\|R\|\geq \rho$ for all $R\in \mathcal{R}$.
    \item The length of any $\epsilon$-piece contained in any word $R\in \mathcal{R}$ is smaller than $\mu\|R\|$.
\end{enumerate}
\end{definition}
Suppose now that $G$ be a group defined by 
\begin{align}\label{hyppresent}
G=\langle X|\mathcal{O}\rangle
\end{align}
where $\mathcal{O}$ is the set of all relators (not only defining) of $G$. Given a symmetrized set of words $\mathcal{R}$, we consider the quotient group ,
\begin{align}\label{representation}
H=\langle G|\mathcal{R}\rangle =\langle X|\mathcal{O}\cup \mathcal{R}\rangle.
\end{align} 
A cell over a van Kampen diagram over ($\color{blue}{\ref*{representation}}$) is called an $\mathcal{R}$-cell (respectively, an $\mathcal{O}$-cell) if it's boundary label is a word from $\mathcal{R}$ (respectively, $\mathcal{O}$). We always consider van Kampen diagram over $(\color{blue}{\ref*{representation}})$ up to some elementary transformations. For examples we do not distinguish diagrams if one can be obtained from other by joining two distinct $\mathcal{O}$-cells having a common edge or by inverse transformation, etc (ref, \cite[Section 5~]{Ol93}).\par 
Let $\triangle$ be a van Kampen diagram over ($\color{blue}{\ref*{representation}}$), $q$ be a sub-path of it's boundary $\partial \triangle$, $\Pi,\Pi'$ some $\mathcal{R}$-cells of $\triangle$. Suppose $p=s_1q_1s_2q_2$ be a simple closed path in $\triangle$, where $q_1$ (respectively $q_2$) is a sub-path of the boundary $\partial \Pi$ (respectively $q$ or $\partial \Pi'$) with $\max\{\|s_1\|,\|s_2\| \}\leq \epsilon $ for some constant $\epsilon$. Then denote $\Gamma$ to be the sub-diagram of $\triangle$ bounded by $p$. We call $\Gamma$ is an $\epsilon$-contiguity sub-diagram of $\Pi$ to the part $q$ of $\partial \triangle$ (or $\Pi'$ respectively) if $\Gamma$ contains no $\mathcal{R}$-cells. The sub-paths $q_1,q_2$ are called contiguity arcs of $\Gamma$ and the ratio $\|q_1\|/\|\partial \triangle\|$ is called contiguity degree of $\Pi$ to $\partial \triangle$ (or to $\Pi'$ respectively). Contiguity degree is denoted by $(\Pi,\Gamma,\partial\triangle)$ or $(\Pi,\Gamma,\Pi')$.\\
We call a (disc) van Kampen diagram over $(\color{blue}{\ref*{representation}})$ minimal if it has minimal number of $\mathcal{R}$-cells among all disc diagrams with the same boundary label.

\begin{lemma}[lemma 4.6, \cite{OOS07}]
	Let $G$ be a $\delta$ hyperbolic group having presentation $\langle X|\mathcal{O}\rangle$ as $(\color{blue}{\ref*{hyppresent}})$, $\epsilon \geq 2\delta$, $0<\mu\leq 0.01$, and $\rho$ is large enough (it suffices to choose $\rho > 10^6\frac{\epsilon}{\mu}$). Let $H$ is given by,
	\begin{align}
	H=\langle G|\mathcal{R}\rangle =\langle X|\mathcal{O}\cup \mathcal{R}\rangle.
	\end{align}
	as in ($\color{blue}{\ref*{representation}}$) where $\mathcal{R}$ is a symmetrized set of words in $X^{\pm 1}$ satisfying the $C(\epsilon,\mu,\rho)$-condition. Then the following holds,
	
	1. Let $\triangle$ be a minimal disc diagram over ($\color{blue}{\ref*{representation}}$). Suppose that $\partial\triangle=q^1q^2\cdots q^t$, where the labels of $q^1,q^2,\cdots ,q^t$ are geodesic in $G$ and $t\leq 12$. Then provided $\triangle $ has an $\mathcal{R}$-cell, there exists an $\mathcal{R}$-cell $\Pi$ in $\triangle$ and disjoint $\epsilon$-contiguity sub-diagrams $\Gamma_1,\Gamma_2,\cdots,\Gamma_t$ (some of them may be absent) of $\Pi$ to $q^1,\cdots,q^t$ respectively such that, $$(\Pi,\Gamma_1,q^1)+\cdots +(\Pi,\Gamma_t,q^t)>1-23\mu.$$

	2. $H$ is a $\delta_1$ hyperbolic group with $\delta_1\leq 4L$, where $L=\max\{\|R\| | R\in \mathcal{R} \}$.
	
\end{lemma}

%%%%%%%%%%%%%%%%%%%%%%%%%%%%%%%%%%%%%%%%%%%%%%%%%%%%%%%%%%%%%%%%%%%%%%%%%%%%%%%%%%%%%%%%%%%%%%%%%%%%%%%%%%%%%%%%%%%%%%%%%%%%%%%%%%%%%%%%%%%%%%%%%%%%%%%%%%%%%%%%%%%%%%%%%%%%%%%%%%%%%%%%%%%%%%%%%%%%%%%%%%%%%%%%%%%%%%%%%%%%%%%%%%%%%%%%%%%%%%%%%%

\subsection{Elementary subgroups of hyperbolic groups}

A group $E$ is called elementary if it is virtually cyclic. We now state an elementary properties of elementary group.
\begin{lemma}
	If $E$ is a torsion free elementary group then $E$ is cyclic.
\end{lemma}

\begin{lemma}
	Let $E$ be an infinite elementary group. Then it contains normal subgroups $T\leq E^{+}\leq E$ such that $|E:E^{+}|\leq 2$ , $T$ is finite and $E^{+}/T\simeq \mathbb{Z}$. If $E\neq E^+$ then $E/T\simeq D_{\infty}$(infinite dihedral group).
\end{lemma}
\begin{definition}
	Let $G$ be a hyperbolic group and $g\in G$ be an infinite order element. Then the elementary subgroup containing $g$ is equal to the following set,
	\begin{center}
		$E(g):=\{x\in G| \ x^{-1}g^nx=g^{\pm n} \ for\ some \ n=n(x)\in \mathbb{N}-\{0\} \}$
	\end{center}
\end{definition}
\begin{remark} 
	For hyperbolic group $E(g)$ is unique maximal elementary subgroup of $G$ containing the infinite order element $g\in G$ (see \cite[Lemma 1.16]{Ol93}). Geometrically $E(g)$ is the kernel of the natural action of non elementary hyperbolic group $G$ on its hyperbolic boundary.
\end{remark}

%%%%%%%%%%%%%%%%%%%%%%%%%%%%%%%%%%%%%%%%%%%%%%%%%%%%%%%%%%%%%%%%%%%%%%%%%%%%%%%%%%%%%%%%%%%%%%%%%%%%%%%%%%%%%%%%%%%%%%%%%%%%%%%%%%%%%%%%%%%%%%%%%%%%%%%%%%%%%%%%%%%%%%%%%%%%%%%%%%%%%%%%%%%%%%%%%%%%%%%%%%%%%%%%%%%%%%%%%%%%%%%%%

\subsection{Necessary lemmas and theorems}
\begin{lemma}\cite{BH99}\label{Hau}
	Let $G=\langle X\rangle $ be a $\delta$ hyperbolic group. Then there exists a constant $R_{\lambda,c}$ depending on $\lambda,c$ such that any $(\lambda,c)$ quasi geodesic in $\Gamma(G,X)$ is $R_{\lambda,c}$ Hausdorff distance away from a geodesic. 
\end{lemma}
The following theorem can also be found as a corollary of the combination theorem in \cite{BF93}. 
\begin{theorem}\cite[Corollary 7]{MO98}\label{olmik}
	Let $G$ and $H$ be hyperbolic groups, $A$ and $B$ be infinite elementary subgroups of $G$, $H$ respectively. Then the free product of the groups $G$ and $H$ with amalgamated subgroups $A$ and $B$ is hyperbolic if and only if either $A$ is a maximal elementary subgroup of $G$ or $B$ is a maximal elementary
	subgroup of $H$.
\end{theorem}
\begin{theorem}\cite[Lemma 6.7, 7.5]{Ol93}\label{ol93lemma}
	Let $H_1,H_2,\cdots,H_k$ be non elementary subgroups of a hyperbolic group $G$ and $\lambda>0$. Then there exists $\mu>0$ such that for any $c\geq 0$, there is $\epsilon\geq 0$ such that for any $N\geq 0$ there is $\rho>0$ with the following property:\par 
	Let the symmetrized presentation in (\color{blue}\ref{representation}\color{black})\ satisfies $C'(\lambda,c,\epsilon,\mu,\rho)$ condition. Then the quotient $H$ is a hyperbolic group. Moreover, $W=_H1$ if and only if $W=_G1$ for every word $W$ with $\|W\|\leq N$ and the images of $H_1,H_2,\cdots,H_k$ are non elementary subgroups in the quotient group $H$. 
\end{theorem}
%\begin{theorem}[\cite{MO98}]\label{olmik}
%	Let $G$ be hyperbolic group with two isomorphic infinite elementary subgroups $A,B$ and let $\phi$ be an isomorphism from $A$ to $B$. Then the HNN-extension $H=<G,t|t^{-1}at=\phi(a),a\in A>$ of $G$ associated with $A$ and $B$ is hyperbolic if and only if the following two conditions hold:\\
%	1. either $A$ or $B$ is a maximal elementary subgroup of $G$.\\
%	2. for all $g\in G$ the subgroup $gAg^{-1}\cap B$ is finite. (for torsion free case finite can be replaced by trivial).
%\end{theorem}
We are going to state a simple version of a lemma that we will be using later
\begin{lemma}[Lemma 7.5, \cite{Dar17}]\label{properpower}
	Let $H=\la X\ra $ be a torsion free non elementary hyperbolic group and $G=H/\la\la \mathcal{R}\ra\ra$ be a quotient group where $\mathcal{R}$ satisfies $C'(\lambda,c,\epsilon,\mu,\rho)$ condition for spars enough parameters $\lambda,c,\epsilon,\mu,\rho$ with $1>1-122\lambda\mu>0$. Suppose $U,W\in \mathcal{F}(X)$ such that $U=_GW^k$ for some $k\geq 2$ and $$\|U\| <\frac{\mu\rho-c}{\lambda}-\epsilon.$$ Then $U=_HW^k$.  
\end{lemma}

%%%%%%%%%%%%%%%%%%%%%%%%%%%%%%%%%%%%%%%%%%%%%%%%%%%%%%%%%%%%%%%%%%%%%%%%%%%%%%%%%%%%%%%%%%%%%%%%%%%%%%%%%%%%%%%%%%%%%%%%%%%%%%%%%%%%%%%%%%%%%%%%%%%%%%%%%%%%%%%%%%

\subsection{Words with small cancellations}
Let $G=\la X\ra$ be a non elementary torsion free $\delta$-hyperbolic group for some $\delta>0$. Let us consider the set $\mathcal{R}$ of words consisting of the form,
\begin{align}\label{systemwords}
R_i=z_iU^{m_{i,1}}VU^{m_{i,2}}V\cdots VU^{m_{i,j_i}} , \ i=1,2,\cdots,k
\end{align} 
and their cyclic shifts, where $k\in\mathbb{N},U,V,z_1,z_2,\cdots,z_k\in \mathcal{F}(X)$ are geodesic words in $G$, $U,V\neq_G 1$ and $m_{i,t}\in\mathbb{N}$ (also assume that $m_{i,j}\neq m_{i',j'}$ if $(i,j)\neq (i',j')$) for $1\leq i\leq k,1\leq t\leq j_i$. Denote $Z=\{ z_1,z_2,\cdots ,z_k\}$. Let $L:=\max \{\|U\|,\|V\|,\|z_1\|,\|z_2\|,\cdots,\|z_k\| \}$, $\underline{m}:=\min\{m_{i,t} |1\leq i\leq k,1\leq t\leq j_i \}$ and $\overline{m}_i:=\max\{m_{i,t}|1\leq t\leq j_i \}$.\par 
We also specify $m_{i,t}$ as follows: choose $m_{1,1}$ and for all $1\leq i\leq k, \ m_{i,1}=2^{i-1}m_{1,1}, j_i=m_{i,1}-1$ and for all $1\leq t\leq j_i, \ m_{i,t}=m_{i,1}+(t-1)$.
\begin{lemma}\cite[Lemma 5.1]{Dar17}\label{smallwords}
	For the set of words $\mathcal{R}$ suppose that $V\notin E(U)$, $z_i\notin E(U)$ for
	$1 \leq i \leq k$. Then there exist constants $\lambda = c = \tilde{K} \in\mathbb{N}$ depending on $G, U, V$ and $Z$, such that the words of the system (\color{blue}\ref{systemwords}\color{black}) are $(\lambda, c)$-quasi-geodesic in
	$\Gamma(G, X)$, provided that $\underline{m} \geq \tilde{K}$ .
\end{lemma}

Let $\lambda,c$ is defined from Lemma \color{blue}\ref{smallwords}\color{black}, $\mathcal{R}$ is defined as above. Since all the words in the system of words $\mathcal{R}$ is $(\lambda,c)$-quasi geodesic by Lemma \color{blue}\ref{smallwords}\color{black}, there exists a constant $R_{\lambda,c}$ such that these words are $R_{\lambda,c}$ Hausdorff distance away from a geodesic by Lemma \color{blue}\ref{Hau}\color{black}. Also assume that with respect to given constants $\epsilon\geq 0,\mu>0,\rho>0$, following conditions are satisfied. 
\begin{align}\label{allconditions}
\|R\|\geq \rho, \ for \ all\ R\in\mathcal{R}\nonumber\\
\underline{m}\geq\tilde{K}\\
\mu\rho\geq 6L(\overline{m}_i+1) \ for \ 1\leq i\leq k\nonumber \\
\underline{m}\geq \frac{2\epsilon'}{\|U\|}\cdot 12\lambda,\ where\ \epsilon'=\epsilon +2\|c\|+5(2R_{\lambda,c}+182\delta +\frac{\|c\|}{2})\nonumber 
\end{align}
%Since all the words in the system of words $\mathcal{R}$ is $(\lambda,c)$-quasi geodesic by Lemma \color{blue}\ref{smallwords}\color{black}, there exists a constant $R_{\lambda,c}$ such that these words are $R_{\lambda,c}$ Hausdorff distance away from a geodesic by Lemma \color{blue}\ref{Hau}\color{black}.
\begin{lemma}\cite[Lemma 5.2]{Dar17}\label{lemma5}
	Using the setting of the previous lemma and assuming that the above
	described conditions take place, let us consider the system of words $\mathcal{R}$ given by (\color{blue}\ref{systemwords}\color{black}). Let $\lambda,c$ be defined by the Lemma \color{blue}\ref{smallwords}\color{black}. Then, if for the given constants $\epsilon \geq 0,
	\mu > 0, \rho > 0$, the conditions in (\color{blue}\ref{allconditions}\color{black}) are satisfied, then the system $\mathcal{R}$ satisfies the $C'(\lambda,c,\epsilon,\mu,\rho)$-condition.
\end{lemma}
For any given constants $\epsilon \geq 0,\mu > 0, \rho > 0$, we denote a system of words $\mathcal{R}$ as described above, by $\mathcal{R}(Z,U,V,\lambda,c,\epsilon,\mu,\rho)$.

%%%%%%%%%%%%%%%%%%%%%%%%%%%%%%%%%%%%%%%%%%%%%%%%%%%%%%%%%%%%%%%%%%%%%%%%%%%%%%%%%%%%%%%%%%%%%%%%%%%%%%%%%%%%%%%%%%%%%%%%%%%%%%%%%%%%%%%%%%%%%%%%%%%%%%%%%%%%%%%%%%%%%%%%%%%%%%%%%%%%%%%%%%%%%%%%%%%%%%%%%%%%%%%%%%%%%%%%%%%%%%%%%%%%%%%%%%%%%%%%%%%%%%%%%%%%

\subsection{Elementary subgroups of lacunary hyperbolic groups}

\begin{definition}
	Let $G$ be a lacunary hyperbolic group and $g\in G$ be an infinite order element. Then define $E^{\mathcal{L}}(g)=\{x\in G |xg^nx^{-1}=g^{\pm n}$, for some $n=n(x)\in \mathbb{N}-\{0\}\}$.
\end{definition}

\begin{theorem}\label{elementarygroup}\label{elementarytheorem}
	Let $G$ be a lacunary hyperbolic group and $g\in G$ be an infinite order element. Then $E^{\mathcal{L}}(g)$ has a locally finite normal subgroup $N\triangleleft G$ such that:\\
	Either $E^{\mathcal{L}}(g)/N$ is an abelian group of Rank 1(i.e. $E^{\mathcal{L}}(g)/N  \ embeds \ in \ (\mathbf{Q},+)$) or $E^{\mathcal{L}}(g)/N$ is an extension of a rank one group by involutive automorphism (i.e, $a\rightarrow a^{-1}$).\\
\end{theorem}

\begin{proof}
	By Theorem~\color{blue}\ref*{OOS}\color{black}, there exist hyperbolic groups $G_i$ and epimorphism $\alpha_i$:\\
	$$G_i\overset{\alpha_i}{\longrightarrow}G_{i+1}\overset{\alpha_{i+1}}{\longrightarrow}G_{i+2}\cdots$$\\
	where $G_i$ is generated by a finite set $\la S_i\ra $ and $\alpha_i(S_i)=S_{i+1}$. Also $G_i$'s are $\delta_i$ hyperbolic where $\delta_i$=$o(r_{S_i}(\alpha_i))$ (where $r_{S_i}(\alpha_i)$=injective radius of $\alpha_i$ w.r.t. $S_i$ ) and $$G=\underrightarrow{\lim} \ G_i$$
	Take a representative\footnote{representative of an element $g\in G$ we mean $\{g_i\}$, $g_i\in G_i$ with $\alpha_i(g_i)=g_{i+1}$ for all $i$} $\{g_i\}_{i\in \mathbb{N}}$ of the element $g$ in the direct limit. Then $g_i$ has infinite order in $G_i$ for all $i\in \mathbb{N}$ as $g\in G$ has infinite order. So for all $i\in \mathbb{N}$, there exists a unique maximal elementary subgroup of $G_i$ containing $g_i$ (as $G_i$ is hyperbolic), namely $E(g_i)$.
	
	\begin{lemma}\label{lemma1}
		Suppose $H$ and $K$ are two hyperbolic groups and $\alpha :H\longrightarrow K$ be an epimorphism. If $h\in H$ and $\alpha(h)\in K$ has infinite order then $$\alpha(E(h))\subset E(\alpha(h))$$ where $E(h)$ and $E(\alpha(h))$ are unique maximal elementary subgroup containing $h$ and $\alpha(h)$ respectively.
	\end{lemma}
	
	\begin{proof}
		From a characterization of the unique maximal elementary subgroup of hyperbolic group we can write, $E(h)=\{x\in H|xh^nx^{-1}=h^{\pm n}$, for some $n=n(x)\in \mathbb{N}-\{0\}\}$ and similarly for $E(\alpha(h))=\{x\in K|x\alpha(h)^mx^{-1}=\alpha(h)^{\pm m}$,for some $m=m(x)\in \mathbb{N}-\{0\}\}$. So the lemma clearly follows for this. 
	\end{proof}
	
	By previous lemma we can see that $\alpha_i(E(g_i))\subset E(g_{i+1})$. Define $\overline{E}(g)=\underrightarrow{\lim}E(g_i)$. Next lemma shows that this definition is independent of choice of that representative $\{g_i\}_{i\in \mathbb{N}}$.
	
	\begin{lemma}\label{lemma2}
		$(a)$ $\overline{E}(g)$ is independent of choice of $\{g_i\}_{i\in \mathbb{N}}$.\\
		$(b)$ $\overline{E}(g)=E(g)$.
	\end{lemma}

	\begin{proof}
		$(a):$ Suppose $\{g'_j\}_{j\in \mathbb{N}}$ be another representative of $g$ in the direct limit. Then define $\overline{E'}(g)=\underrightarrow{\lim}E(g'_i)$. Let $x\in \overline{E'}(g)$ and $\{x_j\}$ be a representative of $x$ in direct limit. Hence $x_{j_0}\in E(g'_{j_0})$ for some $j_0$ hence for all $j\geq j_0$. But $\{g_i\}_{i\in \mathbb{N}}$ and $\{g'_j\}_{j\in \mathbb{N}}$ are both representative of the element $g$, So there exists a number $k\in \mathbb{N}$ such that $g'_k=g_k$. Hence $x_t\in E(g'_t)=E(g_t)$ for all $t\geq k$. Which implies $x\in \overline{E}(g)$. So we have $\overline{E'}(g)\subset \overline{E}(g)$. Now by doing the same method we can show that $\overline{E}(g)\subset \overline{E'}(g)$. That proves $(a)$.\\
		\\ 
		$(b):$ From the definition of $E^{\mathcal{L}}(g)$ and $\overline{E}(g)$ it is clear that $\overline{E}(g)\subset E_G(g)$. Now to prove the other inclusion let assume that $x\in E(g)$. Then there exists representative $\{x_i\}_{i\in \mathbb{N}}$ for $x$ and $\{g_i\}_{i \in \mathbb{N}}$ for $g$ and natural numbers $s,n$ such that:
		$$x_sg_{s}^nx_s^{-1}=g_s^{\pm n}$$
		Hence $x_s\in E(g_s)<G_s $ for large enough $s(\geq j_0)$ by the characterization we mentioned in the proof of lemma $\color{blue}{\ref*{lemma1}}$. Hence $x\in \underrightarrow{\lim}E(g_i)$ (But this direct limit is independent of choice of representative of $g$ by part $(a)$). So $x\in \overline{E}(g)$. Hence we have $\overline{E}(g)=E(g)$.	
	\end{proof}
	
	Now define $E^{+}(g)=\{x\in G|xg^nx^{-1}=g^n$ for some $n=n(x)\in \mathbb{N}-\{0\}\}$. Then clearly $E^+(g)<E(g)$.\\
	Take a representative $\{g_i\}_{i\in \mathbb{N}}$ of the element $g$ and define $E^+(g_i)=\{x\in G_i|xg_i^mx^{-1}=g_i^m$, for some positive integer $m=m(x)\}$. Clearly $E^+(g_i)<E(g_i)$ and $|E(g_i):E^+(g_i)|\leq 2$ for every $i\in\mathbb{N}$.\\
	Now from lemma $\color{blue}{\ref*{lemma1}}$ and lemma $\color{blue}{\ref*{lemma2}}$ it follows that $E^+(g)$ is independent of choice of representative of $g$ and $E^+(g)=\underrightarrow{\lim}E^+(g_i)$.\\
	Hence it follows from the definition that $|E(g):E^+(g)|\leq 2$.
	As we can see that $E(g)$ and $E^+(g)$ is independent of choice of representative of the element $g$, we can fix a representative $\{g_i\}_{i\in \mathbb{N}}$ of $g$. Define $T_i<E^+(g_i)$ be the set of all torsion elements of $E^+(g_i)$. Now $T_i$ is finite normal subgroup of $E(g_i)$ for every $i\in \mathbb{N}$.\\
	In particular for every $i\in \mathbb{N}$ we get a series of normal subgroups of $G_i$:\\
	$$1\leq T_i\leq E^+(g_i)\leq E(g_i),$$ where $|E^+(g_i):E(g_i)|\leq 2$ and $E^+(g_i)/T_i$ is infinite cyclic and $T_i$ is finite.\\
	\\
	Now we can see that $\alpha_i(T_i)\subset T_{i+1}$.(Follows from the definition that $T_i$ is torsion subgroup of $E^+(g_i)$ and $\alpha_i(E^+(g_i))\subset E^+(g_{i+1})$) So we define $T=\underrightarrow{\lim}T_i$, the set of torsion elements of $E^{\mathcal{L}}(g)$. And this definition of $T$ is also independent of choice of representative of the element $g$.

	\begin{lemma}
		$T$ is locally finite normal subgroup of $E(g)$.
	\end{lemma}
	
	\begin{proof}
		A direct limit of finite groups is always locally finite and a limit of normal subgroups is always normal.
	\end{proof}
	
	\begin{lemma}
		$E^+(g)/T$ is an abelian group of rank 1.
	\end{lemma}
	
	\begin{proof}
		Let $[x],[y]\in E^+(g)/T$ and $x,y\in G$ with representative $\{x_i\}$ and $\{y_i\}$ respectively. Let $\{g_i\}$ be a representative of $g$. Then there exists $i$ such that $x_i,y_i\in E^+(g_i)$ in $G_i$, i.e. $x_i^{-1}y_i^{-1}x_iy_i\in T_i$. Hence $[x][y]=[y][x]$ in $E^+(g)/T$ in $G$, i.e. $E^+(g)/T$ is abelian. 
		\\ 
		Let $[x],[y]\in E^+(g)/T$. As this group is abelian we get $x_iy_i=y_ix_i$ (mod T) in $G_i$ for some $i\in \mathbb{N}$, where $\{x_i\},\{y_i\}$ are representatives of $x,y$ respectively. Hence if $y_i\notin T_i$ we have $x_i\in E(y_i)$. So there exists integers $m,n\neq 0$ such that $x_i^my_i^n=1_{G_i}$. Hence $x^ny^m=1_G$. This proves the lemma hence the proposition.
	\end{proof}
\end{proof}

\begin{proposition}
	Let $g$ be an infinite order element in a lacunary hyperbolic group $G$. Every finitely generated subgroup of $E^{\mathcal{L}}(g)$ is elementary. 
\end{proposition}

\begin{proof}
	Let $H=\la a_1,a_2,\cdots ,a_k\ra$ be a finitely generated subgroup of $E^{\mathcal{L}}(g)$. Also let $G=\la x_1,\cdots ,x_n\ra $. By definition of $E^{\mathcal{L}}(g)$ there exists natural numbers $\{l_i\}_{i=1}^{k}$ such that 
	\begin{align}\label{elementaryrel} 
	a_ig^{l_i}a^{-1}_i=g^{\pm l_i} \ for\  all \ i=1,2,\cdots ,k
	\end{align}
	So there exists a natural number $n_0$ such that (\ref{elementaryrel}) is satisfied in $G_j$ for all $j\geq n_0$ (where $G_i$'s are from \ref{lacunarylim} of Theorem\ref{OOS}). One can see that $H\leq E(g)$ in $G_{n_0}$. Hence $H$ is virtually cyclic.
\end{proof}
Converse of the Theorem \color{blue}\ref{elementarytheorem}\color{black}, is also true in general. We are going to see torsion free version of the converse now with some interesting corollary.  
\begin{theorem}\label{rank1}
	Given any non trivial rank one abelian group $L$ and a non elementary hyperbolic group $G_0$, there exists a lacunary hyperbolic quotient $G$ of $G_0$ and an infinite order element $g\in G$ such that $E^{\mathcal{L}}(g)$ is isomorphic to $L$.
\end{theorem}

\begin{proof}
	To construct such group we are going to use the following method (section 11,12 of \cite{Dar17}). We are going to construct the following chain;\\
	\begin{align}
	G_0\overset{\beta_0}{\hookrightarrow} H_1\overset{\gamma_1}{\twoheadrightarrow} G_1\overset{\beta_1}{\hookrightarrow} H_2 \overset{\gamma_1}{\twoheadrightarrow} G_2  \cdots
	\end{align}
	where $H_i,G_i$ are hyperbolic for all $i$ and $\gamma_i\circ \beta_{i-1}$ is surjective for all $i\geq 1$. \par 
	Being a rank $1$ abelian group $L$ can be written as $L=\cup_{i=0}^{i=\infty}L_i$, where $L_i=\langle g_i\rangle_{\infty}$ and $g_i=(g_{i+1})^{m_{i+1}}$ for some $m_{i+1}\in \mathbb{N}$.\par 
	Let $G=\langle X|\mathcal{O}\rangle$ be a torsion free non elementary hyperbolic group with $a,b\in X$ and $h\in G\setminus\{e\} $, not a proper power in $G$ such that $a,b\notin E(h),b\notin E(a)$. We define,
	\begin{align}\label{amalgamprod}
	H:=G*_{h=(g')^m}\langle g'\rangle_{\infty}
	\end{align}
	where $\langle g'\rangle_{\infty}$ is an infinite cyclic group and $m\in\mathbb{N}$. \par 
	Then we shall construct a non elementary hyperbolic factor group $G'$ such that ,
	\begin{align}
	G\overset{\beta}{\hookrightarrow} H\overset{\gamma}{\twoheadrightarrow} G'
	\end{align} 
	Where $\beta$ is an embedding induced from $Id:X\rightarrow X$ and $\gamma\circ \beta$ is a surjective homomorphism. \par 
	By Theorem \color{blue}\ref{olmik}\color{black}, the group $H$ in (\color{blue}\ref*{amalgamprod}\color{black}) is hyperbolic as $\langle g'\rangle_{\infty}$ is maximal elementary subgroup of $G*\langle g'\rangle_{\infty}$. \par
	Let $N$ be a positive number. For any $\epsilon\geq 0,\mu>0,\rho>0$ with $\lambda,c$ given by Lemma \color{blue}\ref{smallwords}\color{black}\ and $\lambda,\mu,\epsilon,\rho,N,c$ satisfies condition for Theorem \color{blue}\ref{ol93lemma}\color{black}\  for the hyperbolic group $G$. Let  $\mathcal{R}(Z,U,V,\lambda,c,\epsilon,\mu,\rho)$ be a set of words as described in Section \color{blue}4.4\color{black}, with $Z=\{z \}$ such that $z,U,V$ are geodesic words in $\mathcal{F}(X)$ representing the elements $h,a,b$ respectively i.e, $z=_Gh, U=_Ga,V=_Gb$. 
	Combining the fact that $\mathcal{R}(Z,U,V,\lambda,c,\epsilon,\mu,\rho)$ satisfies $C'(\lambda,c,\epsilon,\mu,\rho)$ and Theorem \color{blue}\ref{ol93lemma}\color{black}\ with $H_1:=gp\{a,b\},H_2:=gp\{a,h\},H_3:=gp\{b,h\}$, we get that the factor group $G':=\langle H|\mathcal{R}(Z,U,V,\lambda,c,\epsilon,\mu,\rho)\rangle $ is hyperbolic and the injective radius of the factor homomorphism is $\geq N$. Note that we can choose $N$ as large as we want. Also we get that $\gamma \circ \beta$ is surjective homomorphism. Moreover we get that for $i=1,2,3$, images of $H_i$ are non elementary in the factor group $G'$, i.e. $a,b\notin E(h),a\notin E(b)$ in $G'$. Note that by Lemma \color{blue}\ref{properpower}\color{black},  $g'$ is not a proper power in $G'$ for spars enough parameters $\lambda,c,\epsilon,\mu,\rho$. Also note that this fact can be deduced directly from \cite[Theorem 2, property (5)]{Ol93}. \par 
	
	We start with the given non elementary hyperbolic group $G_0$ and let $F(a,b,h)$ be a free subgroup of $G_0$ with three generator over the alphabets $\{a,b,h\}$. Note that it was shown in \cite[page 157]{Gr87} that every non elementary hyperbolic group contains a copy of free group with two generator hence one can always choose a free subgroup of any rank as a subgroup of a non elementary hyperbolic group. $a,b\notin E(h),b\notin E(a)$ in $G_0$. Now we perform the step described in beginning of the proof with $H:=G_0*_{h=(g_1)^{m_1}}\langle g_1\rangle_{\infty}$ to get the factor group $G_1$ with $a,b\notin E(h),b\notin E(a)$ in $G_1$ and $g_1$ is not a proper power in the factor group. Hence one can now use induction to get $G_i$ from $G_{i-1}$ for all $i\geq 1$. Define $G:=\underset{\rightarrow}{\lim} G_i$. $G$ is lacunary hyperbolic group as injectivity radius of factor homomorphism can be as large as we want in every step. Note that we have $E_i(h)\equiv L_i$ for every $i$ by construction, where $E_i(h)$ is the maximal elementary subgroup containing the infinite order element $h$ in $G_i$. One can see that $E^{\mathcal{L}}(h)=\underset{\rightarrow}{\lim}E_i(h)=\cup_{i=1}^{\infty}L_i$ and hence $E^{\mathcal{L}}(h)\cong L$. 
\end{proof}
\begin{corollary}
	There exists lacunary hyperbolic group $G$ and an infinite order element $g\in G$ such that there is no maximal elementary subgroup of $G$ containing the element $g$.
\end{corollary}

For hyperbolic group G we know that for two infinite order element $a,b\in G$ , $E(a)\cap E(b)$ is finite if $E(a)\neq E(b)$ (ref. \cite{Ol93}). In the case of lacunary hyperbolic group one can obtain the following,
\begin{lemma}
	Let $G=\underset{\rightarrow}{lim} G_i$ be a lacunary hyperbolic group with two infinite order element $a,b\in G$. If $E^{\mathcal{L}}(a)\neq E^{\mathcal{L}}(b)$ then $E^{\mathcal{L}}(a)\cap E^{\mathcal{L}}(b)$ is a locally finite group. 
\end{lemma} 
\begin{proof}
	Let $\{a_i\},\{b_i\} $ be representatives of $a,b$ respectively. Note that for any non trivial element $g\in G$, $E^{\El}(g)=\underset{\rightarrow}{\lim}E_i(g)$, where $E_i(g)$ is the maximal elementary subgroup in $G_i$ containing the element $g$ when viewed as an element of $G_i$ for every $i$. We get $E^{\mathcal{L}}(a)\cap E^{\mathcal{L}}(b)= \underset{\rightarrow}{\lim}E_i(a_i)\cap E_i(b_i) $, where $E_i(a),E_i(b)$ are the maximal elementary subgroups of the hyperbolic group $G_i$ containing the elements $a_i,b_i$ respectively. $E_i(a_i)\cap E_i(b_i) $ is finite (see \cite[Lemma 1.6]{Ol93}). Being a limit of finite group $E^{\mathcal{L}}(a)\cap  E^{\mathcal{L}}(b)$ is locally finite.
\end{proof}

%%%%%%%%%%%%%%%%%%%%%%%%%%%%%%%%%%%%%%%%%%%%%%%%%%%%%%%%%%%%%%%%%%%%%%%%%%%%%%%%%%%%%%%%%%%%%%%%%%%%%%%%%%%%%%%%%%%%%%%%%%%%%%%%%%%%%%%%%%%%%%%%%%%%%%%%%%%%%%%%%%%%%%%%%%%%%%%%%%%%%%%%%%%%%%%%%%%%%%%%%%%%%%%%%%%%%%%%%%%%%%%%%%%%%%%%%%%%%%
\subsection{Rank one abelian subgroups of lacunary hyperbolic groups}\label{rank1monstergroup}

In this section we are going to describe rank one abelian subgroups of lacunary hyperbolic groups and maximal subgroups. 
\begin{definition}
	Two element $g,h$ of infinite order in a (hyperbolic) group is said to be commensurable if $g^k=ah^la^{-1}$ for some non zero integers $k,l$ and some element $a\in G$.
\end{definition} 
By \cite[Lemma 1.6]{Ol93}, $g$ and $h$ are commensurable if and only if the maximal elementary subgroups $E(g),E(h)$ containing two elements $g$ and $h$ respectively, are conjugate.  \par 
Let $G_0=\la X\ra$ be a torsion free $\delta$-hyperbolic group with respect to $X$, where $X=\{x_1,x_2,\cdots,x_n \}$ is a finite generating set. Without loss of generality we assume that $E(x_i)\cap E(x_j)=\{e\}$ for $i\neq j$. Let $X$ be linearly ordered such that $x_i^{-1}<x_j^{-1}<x_i<x_j$ if $i<j$. Let $F'(X)$ denote the set of non empty reduced words on $X$, and $F'(X)=\{w_1,w_2,\cdots \}$ be an enumeration with $w_i<w_j$ for $i<j$ according to the lexicographic order induced from the order on $X$. Note that $w_1=(x_1)^{-1}$ and $w_2=(x_2)^{-1}$. We now order the set $\mathcal S:= F'(X)\times F'(X)$ lexicographically and enumerate them as,
\begin{align}\label{u_kwords} 
\mathcal S=\{(u_1,v_1),(u_2,v_2),\cdots  \}
\end{align}
where for $i<j$ we have $(u_i,v_i)<(u_j,v_j)$. \par 
We are going to construct the following chain;\\
\begin{align}
G_0\overset{\beta_0}{\hookrightarrow} H_1\overset{\alpha_1}{\twoheadrightarrow}G'_1\overset{\gamma_1}{\twoheadrightarrow} G_1\overset{\beta_1}{\hookrightarrow} H_2 \overset{\alpha_2}{\twoheadrightarrow}G'_2\overset{\gamma_1}{\twoheadrightarrow} G_2  \cdots
\end{align}
where $H_i,G_i,G'_i$ are hyperbolic for all $i$ and $\gamma_i\circ \alpha_{i-1}\circ \beta_{i-1}$ is surjective for all $i\geq 1$ and takes generating set to generating set. In particular we are going to show that the following chain
\begin{align}
G_0\overset{\gamma_1\circ \alpha_{1}\circ \beta_{0}}{\twoheadrightarrow}G_1\overset{\gamma_2\circ \alpha_{2}\circ \beta_{1}}{\twoheadrightarrow} G_2\overset{\gamma_3\circ \alpha_{3}\circ \beta_{2}}{\twoheadrightarrow}\cdots G_{i-1}\overset{\gamma_i\circ \alpha_{i}\circ \beta_{i-1}}{\twoheadrightarrow}G_i\rightarrow   \cdots
\end{align}
satisfies part c. of Theorem \color{blue}\ref{OOS}\color{black}.\par 
Let $\mathscr{C}=\{Q^p\}_{p\in\mathbb{N}}$ be a countable family of rank one abelian groups. Then for every $p\in\mathbb{N}$, $Q^j$ can be written as $Q^p=\cup_{i=0}^{i=\infty}Q^p_i$, where $Q^p_i=\langle g^p_i\rangle_{\infty}$ and $g^p_i=(g^p_{i+1})^{m^p_{i+1}}$ for some $m^p_{i+1}\in \mathbb{N}$.\par 
Then there exists a smallest index $j_i\geq i$ such that $v_{j_i}\notin E(u_{j_i})$. For $m\in \mathbb N$, define 
\begin{align}\label{amalgum} 
H^k_{i+1}:=H^{k-1}_{i+1}\underset{u_k=(g^k_{(k,i+1)})^{m^k_{i+1}}}{*}\la g^k_{(k,i+1)} \ra_{\infty}\ where\  H^0_{i+1}=G_i\ and \ g^k_{(k,i+1)}=g^k_{i+1} \ for\ k=1,2,\cdots ,j_i.
\end{align}
For $i \geq 0$ let $H_{i+1}$ to be $H^{j_i}_{i+1}$. Note that $H_{i+1}$ is hyperbolic as $H^{k}_{i+1}$ is hyperbolic for all $k$ by \cite[Theorem 3]{MO98}. By construction there is a natural embedding $\beta_i:G_i\hookrightarrow H_{i+1}$. Take $c_i,c_i'\in G$ such that $c_i',c_i\notin E(u_k)$ for all $1 \leq k\leq j_i$ and $c_i',c_i\notin E(v_{j_i}) $. Such $c_i$ and $c_i'$ exists as there are infinitely many elements in a non elementary hyperbolic group which are pairwise non commensurable by \cite[Lemma 3.2]{Ol93}. Now we define the following set 
\begin{align} 
Y_i:=\{g_{(k,i+1)}|1\leq k\leq j_i \}
\end{align}
Let us consider $\mathcal{\tilde{R}}_{i+1}:=\mathcal{R}(Y_i,c_i,c_i',\lambda,c,\epsilon,\mu,\rho)$ as defined in the section 4.4.  
%\begin{align}\label{Word}
%R_k=g_{(k,i+1)}c_i^{n_{1,k}}c_i'c_i^{n_{2,k}}c'_i\cdots c_i^{n_{s_k,k}}c_i'
%\end{align}
%and $\mathcal{R}$ be the set of all cyclic shifts of of $R_k^{\pm 1}$. choose $n_{s,k}$ as follows,
%$$n_{1,k}=2^{k-1}n_{1,1}, \ for \ 1\leq k\leq j_i, \ s_k=n_{1,k-1} \ and \ n_{s,k}=n_{1,k}+(s-1)$$

Consider the natural quotient map $\alpha_{i+1}:H_{i+1}\twoheadrightarrow G'_{i+1}$ to the quotient $G'_{i+1}:=\la H_{i+1}|\tilde{\mathcal{R}}_{i+1}\ra$. Also the factor group $G'_{i+1}$ is hyperbolic by \cite{Ol93}. it is easy to see that $\alpha_{i+1}\circ \beta_{i}$ is a surjective map that takes generators to generators.\par  
Consider the following set $$Z_i:=\{x\in X|x\notin E(u_{j_i})  \}.$$ 
Let $G_{i+1}:= G'_{i+1}/\la \la \mathcal R (Z_i,u_{j_i},v_{j_i},\lambda,c,\epsilon,\mu,\rho)\ra \ra$ and let $\gamma_{i+1}:G'_{i+1}\twoheadrightarrow G_{i+1}$ be the quotient map. Hence we get that the group $G_{i+1}$ is hyperbolic by Theorem \color{blue}\ref{ol93lemma}\color{black}\ since $\mathcal R (Z_i,u_{j_i},v_{j_i},\lambda,c,\epsilon,\mu,\rho)$ satisfies $C'(\lambda,c,\epsilon,\mu,\rho)$ small cancellation condition as discussed in section \color{blue}4.4\color{black}\ and the map $\gamma_{i+1}$ takes generating set to generating set. In particular $\eta_{i+1}:=\gamma_{i+1}\circ \alpha_{i+1}\circ \beta_{i}$ is surjective homomorphism which takes the generating set of $G_i$ to the generating set of $G_{i+1}$. Let $G^{\mathscr{C}}:=\underset{\rightarrow}{\lim}G_i$. 
We summarize the above discussion in the following statement.
\begin{lemma}\label{limitmax}
	The above construction satisfies the following properties:
	\begin{enumerate}
		\item $G_i$ is non elementary hyperbolic group for all $i$;
		\item Either $u_i\in E(v_i)$ or the group genarated by $\{u_i,v_i \}$ is equal to all of $G_i$;
		\item For every $k\in\mathbb{N}$ we have $E(u_k)\simeq Q^k_l$ in $G_l$ for all $l\geq k$, where the element $u_k$ is defined in $\color{blue}(\ref{u_kwords})\color{black}$. 
		%\item For each element $x\in X$, $E(x)=\la y\ra$ in $G_i$, where $x=y^{m^p_1m^p_2\cdots m^p_i}$. The exponent $m^p_i$'s are described as follows: There exists a $p\in\mathbb{N}$ such that $x=u_p$. Being a rank 1 abelian group $Q^p$ can be written as $Q^p=\cup_{i=0}^{\infty}Q^p_i$, where $Q^p_i=\langle g_i\rangle_{\infty}$ and $g^p_i=({g^p_{i+1}})^{m^p_{i+1}}$ for some $m^p_{i+1}\in \mathbb{N}$. 
	\end{enumerate}
\end{lemma}
\begin{proof} Part 1. follows from \cite[Lemma 7.2]{Ol93}. To see part 2. notice that by definition if $j_i>i$ then $v_i\in E(u_i)$ in $G_i$. Otherwise if $j_i=i$ then $v_i\notin E(u_i)$ in $G_i$ and $G_i=gp\{u_i,v_i\}$. For part 3. note that we have $u_k=(g^{k}_{l})^{m^{k}_l}$ in $H_l$ for all $l\geq k$ by \color{blue}(\ref{u_kwords})\color{black} hence $E(u_k)\simeq Q^k_l$ in $H_l$. By part $(5)$ of \cite[Theorem 2]{Ol93} we get $E(u_k)\simeq Q^{k}_l$ since centralizer of image of an element inside the injectivity radius in the quotient group is same as image of centralizer of that element. For the same reason we have $E(u_k)\simeq Q^{k}_l$ in the quotient group $G_l$.    
\end{proof}
\begin{remark}\label{maxele}
Note that every element $g\in G$ is equal to $u_k$ for some $k\in\mathbb{N}$ and hence $E^{\El}(g)=\underset{\rightarrow}{\lim} E_l(u_k)$ where $E_l(u_k)$ is $E(u_k)$ in $G_l$. By part 3. of lemma \ref{limitmax} we get that $E_l(u_k)\simeq Q^k_l$ for every $l\geq k$. Hence $E^{\El}(g)=\underset{\rightarrow}{\lim} E_l(u_k)=\underset{\rightarrow}{\lim}Q^k_l\simeq Q^k$
%	Throughout this construction we choose such presentation for the given rank 1 abelian group $L$ and fix that throughout the section.
\end{remark}
We now give the main theorem of this section.
\begin{theorem}\label{monster}
	For any torsion free non elementary hyperbolic group $G$ and a countable family $\mathscr{F}:=\{Q^i_m\}_{i\in\mathbb{N}}$ of subgroups of $(\mathbb Q,+)$, there exists a non elementary, torsion free, non abelian lacunary hyperbolic quotient $G^{\mathcal{C}}$ of $G$ such that the set of all maximal subgroups of $G^{\mathscr{F}}$ is equal up to isomorphism to $\{Q^i_m\}_{i\in\mathbb{N}}$ i.e, every maximal subgroup of $G^{\mathscr{F}}$ is isomorphic to $Q^i_m$ for some $i\in\mathbb{N}$ and for every $i$ there exists a maximal subgroup of $G^{\mathscr{F}}$ that is isomorphic to $Q^i_m$.
\end{theorem}
\begin{proof}
	By letting $\mathscr{F}=\mathscr{C}$ in the above construction we get $G^{\mathscr{C}}=G^{\mathscr{F}}$, where $Q^i_j=\langle g^i_j\rangle_{\infty}$ and $g^i_j=({g^i_{j+1}})^{m^i_{j+1}}$ for some $m^i_{j+1}\in \mathbb{N}$ and $Q^i_m=\cup_{j=1}^{\infty}Q^i_j$. One can choose sparse enough parameters to satisfy the injectivity radius condition in Definition \ref{lacunarylim} which in turn will ensure that $G$ is lacunary hyperbolic. By remark \color{blue}\ref{maxele}\color{black}\ we get that for any $g\in G$, $E^{\El}(g)=Q^i$ for some $i$ depending of $g$. Suppose $P\nleqslant G$ is a maximal subgroup of $G$. As $P$ is a proper subgroup, $P$ is abelian by part $2.$ of Lemma \ref{limitmax}. Now let $e\neq h\in G$. Note that being abelian $P$ is contained in the centralizer of $h$. Now from Definition of $E^{\El}(g)$, it follows that $g\in P\leq E^{\El}(g)(\cong Q_m)\nleqslant G$. By maximality of $P$ we get that $P\cong Q_m$. Thus, all maximal subgroups of $G$ are isomorphic to $Q^i_m$ for some $i$ and hence the theorem is proved. 
\end{proof}

As a corollary of this theorem we can recover the seminal result of monster groups constructed by Ol'shanskii in $\cite{Ol93}$.
\begin{corollary}[\cite{Ol93}]
	For every non-cyclic torsion free hyperbolic group $G$, there exists a non abelian torsion free quotient $\overline{G}$ such that all proper subgroups of $\overline{G}$ are infinite cyclic.
\end{corollary}
\begin{proof}Take $Q^i_m=\mathbb{Z}$ for all $i\in\mathbb{N}$ in Theorem \ref{monster}. \end{proof} 
A group $G$ is called divisible if for any element $g$ of $G$ and any non zero integer $n$ the equation $x^n=g$ has a solution in $G$. The first example of non trivial finitely generated divisible group was constructed by V. S. Guba in \cite{Gu86}. Later Ol'shanskii and Mikhajlovskii proved following;
\begin{corollary}\cite[Corollary 3]{MO98}
	For every non cyclic torsion free hyperbolic group $G$ there exists a non abelian torsion free divisible quotient $H$ of $G$.
\end{corollary}
\begin{proof}
	Take $\mathcal{C}=\{\mathbb{Q} \}$ in Theorem \color{blue}\ref{monster}\color{black}.
\end{proof}
\begin{corollary}\cite[Theorem 1]{Gu86}
	There exists a non trivial finitely generated torsion free divisible group. 
\end{corollary}
%\begin{proof}
%	Take $\mathcal{C}=\{\mathbb{Z} \}$ in Theorem \color{blue}\ref{monster}\color{black}.
%\end{proof}

\begin{remark}
	If one start with non trivial property $(T)$ torsion free hyperbolic group in Theorem \color{blue}\ref{monster}\color{black} (same for Corollary 3,\cite{MO98}), then one gets finitely generated non amenable divisible group with property $(T)$, since property $(T)$ is preserved under taking quotients. 
\end{remark} 

\begin{definition} A group $G$ has the unique product property (or said to be unique product group) whenever for all pairs of non empty finite subsets $A$ and $B$ of $G$ the set of products $AB$ has an element $g\in G$ with a unique representation of the form $g=ab$ with $a\in A$ and $b\in B$.
\end{definition}
Note that unique product groups are torsion free. The first examples of torsion-free groups without the unique product property was given by Rips and Segev in \cite{RS87}. Existence of property $(T)$ hyperbolic group without the unique product property has been shown in \cite{AS14}. By starting with property $(T)$ hyperbolic group without unique product property in Theorem \color{blue}\ref{monster}\color{black}, we obtain following:
\begin{corollary}\label{monsteruniqueproduct}
	For every rank one abelian group $Q_m$, there exists a non elementary, torsion free, property $(T)$, lacunary hyperbolic group without the unique product group $G^{Q_m}$ such that any maximal subgroup of $G^{Q_m}$ is isomorphic to $Q_m$.\par 
	In particular, there exist a non trivial property $(T)$ torsion free divisible lacunary hyperbolic group without the unique product property.
\end{corollary}

Note that when we are adding relations to the starting group we need to start with a relation of sufficiently large length (bigger that the size of the sets $A,B$ and $AB$ for with the starting group does not have unique product property).
%%%%%%%%%%%%%%%%%%%%%%%%%%%%%%%%%%%%%%%%%%%%%%%%%%%%%%%%%%%%%%%%%%%%%%%%%%%%%%%%%%%%%%%%%%%%%%%%%%%%%%%%%%%%%%%%%%%%%%%%%%%%%%%%%%%%%%%%%%%%%%%%%%%%%%%%%%%%%%%%%%%%%%%%%%%%%%%%%%%%%%%%%%%%%%%%%%%%%%%%%%%%%%%%%%%%%%%%%%%%%%%%%%%%%%%%%%%%%%%%%%%%%%%%%%%%%%%%%%%%%%%%%%%%%%%%%%
\subsection{Locally finite by rank one abelian subgroups of lacunary hyperbolic groups}\label{rankonebyfinitesection}

In this section we are going to prove the full converse of Theorem \color{blue}\ref{elementarygroup}\color{black}. First we are going to recall some definitions and theorems.
\begin{notation} 
	Given a subgroup $H\leq G$, $H^0$ denotes the set of all infinite order elements of $H$.
\end{notation}
\begin{proposition}[Proposition 1,\cite{Ol93}]
	Let $H$ be a non elementary subgroup of a hyperbolic group $G$. Then $E_G(H):=\underset{x\in H^0}{\bigcap}E(x)$ is the unique maximal finite subgroup of $G$, normalized by the subgroup  $H$, where $E(x)$ is the unique maximal elementary subgroup of $G$ containing the infinite order element $x$.
\end{proposition}

Now we are going to state a well known fact about elementary subgroups of hyperbolic groups.
\begin{proposition}
	For every infinite order element $g$ in a hyperbolic group $G$ there exists a series of normal subgroups;
	\begin{align}
	1\leq T(g)\leq E^{+}(g)\leq E(g)
	\end{align}
	where $T(g)$ is the set of torsion elements of the unique maximal elementary subgroup $E(g)$ of the hyperbolic group $G$, $|E(g):E^+(g)|\leq 2$, $E^+(g)/T$ 
	is infinite cyclic. Note that $T(g)$ is finite.
\end{proposition}
\begin{definition}
	A subgroup $H\leq G$ is called suitable, if there exist two non commensurable elements $g,h$ in $H^0$ such that $E(g)\cap E(h)=\{1\}$. 
\end{definition}
Note that a subgroup $H$ of a hyperbolic group $G$ is same as $H$ is non elementary and $E_G(H)=\{1\}$. 
We would like to state a simple version of a very beautiful theorem by D. Osin in our context,
\begin{theorem}[Theorem 2.4, \cite{Os'10}]\label{osin}
	Let $H$ be a suitable subgroup of a hyperbolic group $G$ and $T:=\{t_1,t_1,\cdots,t_n\}\subset G$. Then there exists $\{w_1,w_2,\cdots,w_n\}\subset H$ such that the quotient group $\overline{G}:=G/\la\la t_1w_1,t_2w_2,\cdots,t_nw_n\ra\ra$ satisfies following; 
	\begin{enumerate}
		\item $\overline{G}$ is hyperbolic.
		\item Image of $H$ in $\overline{G}$ is a suitable subgroup of $\overline{G}$.
		\item If $G$ is torsion free then so is $\overline{G}$.
	\end{enumerate}
\end{theorem}
\begin{remark}\label{osinlacunar}
	Note that one can choose lengths of the elements $w_1,w_2,\cdots,w_n$ to be as large as one wants. 
\end{remark}

\begin{theorem}\label{rankonefinite}
	Let $G$ be a non elementary hyperbolic group, $\{E_i\}_{i\geq 1}$ be a collection of elementary groups with $E_i\subset E_{i+1}$ for all $i\geq 1$ and denote $E:=\underset{i\geq 1}{\bigcup}E_i$. Then there exists a lacunary hyperbolic quotient $\overline{G}$ of $G$ with an infinite order element $g\in \overline{G}$, such that $E^{\El}(g)\simeq E$. 
\end{theorem}
\begin{proof}
	Let $G_0=\la S\ra, \ |S|<\infty $ be a non elementary hyperbolic group, $E_i$ be elementary groups with $\la g_i\ra$ be maximal cyclic subgroup of $E_i$ and $E_i=\la g_i\ra \cup x_1\la g_i\ra \cup \cdots \cup x_{n_i}\la g\ra$ for $i=1,2$ and $n_1\leq n_2$. Note that $x_1,x_2,\cdots,x_{n_2}$ are elements of finite order and $g_1=(g_2)^{m_2}$ for some $m_2\in \mathbb{N}$. Let $H_0$ be a suitable subgroup of $G_0$. One can choose $H_1$ since any non elementary hyperbolic group contains a copy of free group with countably many generators.\par 
	Consider the group, 
	\begin{align}
	G'_0:=(G_0*E_1)\underset{E_1}{*}E_2.
	\end{align}
	Note that $G'_0$ is a non elementary hyperbolic group by Theorem \color{blue}\ref{olmik}\color{black} as $E_1$ is maximal elementary in the hyperbolic group $G_0*E_1$. Note that $H_0$ is a suitable subgroup of $G'_0$. Let $N\in\mathbb{N}$. Now we apply Lemma \color{blue}\ref{osin}\color{black}\ with $T=\{x_1,x_2,\cdots,x_{n_2} \}, \ G=G'_0$ and $H=H_0$. We choose $w_1,w_2,\cdots,w_n$ such that the injectivity radius of the the quotient map $\phi:G'_0\twoheadrightarrow \overline{G'_0}:=G'_0/\la\la x_1w_1,\cdots,x_{n_2}w_{n_2}\ra\ra$ is greater that $N$ (by Remark \color{blue}\ref{osinlacunar}\color{black}). We record the fact that the group $\overline{G_1}$ is generated by $X\cup \{g_2 \}$ and image of the suitable group $H_0$, say $H'_0$ is also suitable in the quotient group $\overline{G'_0}$. Now we consider the set $\mathcal{R}_0$ of cyclic shifts of following set of words
	\begin{align}
	R:=g_2U^{l_1}VU^{l_{2}}V\cdots U^{l_{n_2}}
	\end{align}
	where $U,V$ are geodesic representative of two non commensurable elements $h_1,h_2$ in the suitable group $H'_0$ such that $E(h_1)\cap E(h_2)=\{1\}$ and $E(h_i)=\la h_i\ra_{\infty}$ for $i=1,2$. Then there exists $\lambda,c,N$ such that for all $\epsilon,\mu,\rho$ and $N<l_1<l_2<\cdots<l_{n_2}$, $\mathcal{R}_0$ satisfies $C'(\epsilon,\mu,\lambda,c,\rho)$ condition over the group $\overline{G'_0}$. \par     
	
\begin{lemma} 	
The quotient group $G_1:=\la \overline{G'_0}|\mathcal{R}_0\ra $ enjoys following properties:
	\begin{enumerate}
		\item $G_1$ is hyperbolic.
		\item Injectivity radius of the quotient map $\epsilon: G_0\twoheadrightarrow G_1$ is $\geq N$.
		\item $\epsilon(g_2)$ is not a proper power in $G_1$.
		\item $\epsilon(H_0)$ is again suitable in $G_1$.
		\item $E(\epsilon(g_2))\simeq E_2$.
	\end{enumerate}
\end{lemma} 
\begin{proof}    
$1.$ and $2.$ and $4.$ follows from Theorem \color{blue}\ref{osin}\color{black} and Lemma \color{blue}\ref{ol93lemma}\color{black}.\par 
By property $(5)$ of \cite[Theorem 2]{Ol93} we get that the centralizer $C_{G_1}(\epsilon(a))$ for every $a$ in the injectivity ball of $\epsilon$ is the image of centralizer $\epsilon(C_{G_0}(a))$. Suppose $\epsilon(g_2)=\epsilon(y)^m=\epsilon(y^m)$ for some $y\in G_0$. Then $\epsilon(y)\in C_{G_1}(\epsilon(g_2))=\epsilon(C_{G_0}(g_2))$. This implies that $y\in C_{G_0}(g_2)\ \Rightarrow y\in \la g_2\ra$, as $y$ has infinite order. Hence $m=1$ and $\epsilon(g_2)$ is not a proper power in $G_1$.   \par 
$5.$ follows from the construction and above properties. 
\end{proof}
We start with a given non elementary hyperbolic group $G$. By \cite[Proposition 1]{Ol93} Every non–elementary hyperbolic group $G$ contains a unique maximal normal finite subgroup $K \leq  G$, in fact $K$ is precisely the kernel of the $G$-action on the boundary of $G$. Thus passing to the quotient $G/K$ if necessary we may assume that $G$ has no nontrivial finite normal subgroups. Take $G=G_0$ in the above process to obtain a quotient hyperbolic group $G_1$ of $G$  with an element $g_2$ such that $E(g_2)\simeq E_2$ with $N_1$ the injectivity radius of $\epsilon_1:G\twoheadrightarrow G_1$. Now one can repeat above procedure with a slide modification: take $G'_0=G_1*_{E(g_2)}E_{3}$. Then called the quotient group $G_3$ and the quotient map $\epsilon_2:G_1\twoheadrightarrow G_3$ and the injectivity radius of the map $\epsilon_2$ is $N_3$. We get $E(\epsilon_2(g_2))\simeq E_3$. Using induction we get,     
\begin{align}
G=G_0\overset{\epsilon_1}{\twoheadrightarrow} G_1\overset{\epsilon_2}{\twoheadrightarrow} G_3\overset{\epsilon_3}{\twoheadrightarrow} G_4\overset{\epsilon_4}{\twoheadrightarrow}\cdots 
\end{align}
Define $\overline{G}:=\underset{\rightarrow}{\lim}G_i$. Note that one can take the numbers $N_i$ to be as large as one wants in order to make $\overline{G}$ a lacunary hyperbolic group. Observe that $E^{\El}(g_2)=\underset{i\geq 1}{\bigcup} E_i=E$. Note that we are viewing $g_2$ as an element of the limit group.
\end{proof}

\begin{notation}\label{notationelementary}
	We denote the class of increasing union of elementary subgroups as ${}_{rk-1}\mathscr{E}_{F}$. Note that increasing union of elementary groups and also one can write those as rank one by finite groups. 
\end{notation}

By using the fact that $G$ is countable and by doing the process described in Theorem \color{blue}\ref{rankonefinite}\color{black} we can get following,

\begin{theorem}\label{rk1byfinite}
	Let $G$ be a torsion free non elementary hyperbolic group and $\mathscr{C}:=\{E^j \}_{j\in\mathbb{N}}$ be a countable collection of groups with $E^j\in {}_{rk-1}\mathscr{E}_{F}$ for all $j\geq 1$. Then there exists a non elementary lacunary hyperbolic quotient $G^{\mathscr{C}}$ of $G$ such that $\{E^{\El}(h)\ | \ h\in (G^{\mathscr{C}})^0 \}=\mathscr{C}$.\par	Moreover $\mathscr{C}$ is the set of all maximal proper subgroups of the group $G^{\mathscr{C}}$.
\end{theorem} 
\begin{proof}
Proof of this theorem is basically a suitable combination of Theorem \color{blue}\ref{rankonefinite}\color{black} and the construction in in Section \color{blue}\ref{rank1monstergroup}\color{black}. Let $G=\la X\ra$ be the given non elementary torsion free $\delta$-hyperbolic group with respect to $X$, where $X=\{x_1,x_2,\cdots,x_n \}$ is a finite generating set. Without loss of generality we assume that $E(x_i)\cap E(x_j)=\{e\}$ for $i\neq j$. Let $X$ be linearly ordered such that $x_i^{-1}<x_j^{-1}<x_i<x_j$ if $i<j$. Let $F'(X)$ denote the set of non empty reduced words on $X$, and $F'(X)=\{w_1,w_2,\cdots \}$ be an enumeration with $w_i<w_j$ for $i<j$ according to the lexicographic order induced from the order on $X$. Note that $w_1=(x_1)^{-1}$ and $w_2=(x_2)^{-1}$. We now order the set $\mathcal S:= F'(X)\times F'(X)$ lexicographically and enumerate them as,
$$\mathcal S=\{(u_1,v_1),(u_2,v_2),\cdots  \}$$
where for $i<j$ we have $(u_i,v_i)<(u_j,v_j)$. \par 
We are going to construct the following chain;\\
\begin{align}
G_0\overset{\beta_0}{\hookrightarrow} H_1\overset{\alpha_1}{\twoheadrightarrow}G'_1\overset{\gamma_1}{\twoheadrightarrow} G_1\overset{\beta_1}{\hookrightarrow} H_2 \overset{\alpha_2}{\twoheadrightarrow}G'_2\overset{\gamma_1}{\twoheadrightarrow} G_2  \cdots
\end{align}
where $H_i,G_i,G'_i$ are hyperbolic for all $i$ and $\gamma_i\circ \alpha_{i-1}\circ \beta_{i-1}$ is surjective for all $i\geq 1$ and takes generating set to generating set. In particular we are going to show that the following chain
\begin{align}
G_0\overset{\gamma_1\circ \alpha_{1}\circ \beta_{0}}{\twoheadrightarrow}G_1\overset{\gamma_2\circ \alpha_{2}\circ \beta_{1}}{\twoheadrightarrow} G_2\overset{\gamma_3\circ \alpha_{3}\circ \beta_{2}}{\twoheadrightarrow}\cdots G_{i-1}\overset{\gamma_i\circ \alpha_{i}\circ \beta_{i-1}}{\twoheadrightarrow}G_i\rightarrow   \cdots
\end{align}
satisfies part c. of Theorem \color{blue}\ref{OOS}\color{black}.\par 
We have $\mathscr{C}=\{E^p\}_{p\in\mathbb{N}}$, a countable family of groups from ${}_{rk-1}\mathscr{E}_{F}$ . Then for every $p\in\mathbb{N}$, $E^j$ can be written as $E^p=\cup_{i=0}^{i=\infty}E^p_i$, where $E^p_i=\langle g^p_i\rangle_{\infty}\cup a^p_1\langle g^p_i\rangle_{\infty}\cup a^p_2\langle g^p_i\rangle_{\infty}\cup\cdots\cup a^p_{n_i}\langle g^p_i\rangle_{\infty}$ with order of $a^p_i$ is finite, $n_i\leq n_{i+1}$ for every $i\geq 1$  and $g^p_i=(g^p_{i+1})^{m^p_{i+1}}$ for some $m^p_{i+1}\in \mathbb{N}$.\par 
Then there exists a smallest index $j_i\geq i$ such that $v_{j_i}\notin E(u_{j_i})$. For $m\in \mathbb N$, define 
\begin{align}\label{amalgumrankonebyfinite} 
H^k_{i+1}:=H^{k-1}_{i+1}\underset{u_k=(g^k_{(k,i+1)})^{m^k_{i+1}}}{*} E^k_i \ ,\ where\  H^0_{i+1}=G_i\ and \ g^k_{(k,i+1)}=g_{i+1} \ for\ k=1,2,\cdots ,j_i.
\end{align}
For $i \geq 0$ let $H_{i+1}$ to be $H^{j_i}_{i+1}$. Note that $H_{i+1}$ is hyperbolic as $H^{k}_{i+1}$ is hyperbolic for all $k$ by \cite[Theorem 3]{MO98}. By construction there is a natural embedding $\beta_i:G_i\hookrightarrow H_{i+1}$. Take $c_i,c_i'\in G$ such that $c_i',c_i\notin E(u_k)$ for all $1 \leq k\leq j_i$ and $c_i',c_i\notin E(v_{j_i}) $. Such $c_i$ and $c_i'$ exists as there are infinitely many elements in a non elementary hyperbolic group which are pairwise non commensurable by \cite[Lemma 3.2]{Ol93}. Now we define the following set 
\begin{align} 
Y_i:=\{g_{(k,i+1)}|1\leq k\leq j_i \}
\end{align}
Let us consider $\mathcal{R}(Y_i,c_i,c_i',\lambda,c,\epsilon,\mu,\rho)$ as defined in the section \color{blue}4.4\color{black} and apply Theorem \color{blue}\ref{osin}\color{black}\ with $G:=H_{i+1}/\mathcal{R}(Y_i,c_i,c_i',\lambda,c,\epsilon,\mu,\rho)$,  $T=\{a^s_{l}|1\leq s\leq j_i, 1\leq l\leq n_{j_i} \}$ and suitable subgroup $H_{i}$.  Let $\tilde{\mathcal{R}}_{i+1}= \mathcal{R}(Y_i,c_i,c_i',\lambda,c,\epsilon,\mu,\rho)\cup \{a^s_{l}w^s_{l}|1\leq s\leq j_i, 1\leq l\leq n_{j_i} \}$ \par

Consider the natural quotient map $\alpha_{i+1}:H_{i+1}\twoheadrightarrow G'_{i+1}$ to the quotient $G'_{i+1}:=\la H_{i+1}|\tilde{\mathcal{R}}_{i+1}\ra$. Also the factor group $G'_{i+1}$ is hyperbolic by \cite{Ol93}. it is easy to see that $\alpha_{i+1}\circ \beta_{i}$ is a surjective map that takes generators to generators.\par  
Consider the following set $$Z_i:=\{x\in X|x\notin E(u_{j_i})  \}.$$ 
Let $G_{i+1}:= G'_{i+1}/\la \la \mathcal R (Z_i,u_{j_i},v_{j_i},\lambda,c,\epsilon,\mu,\rho)\ra \ra$ and let $\gamma_{i+1}:G'_{i+1}\twoheadrightarrow G_{i+1}$ be the quotient map. Hence we get that the group $G_{i+1}$ is hyperbolic by Theorem \color{blue}\ref{ol93lemma}\color{black}\ since $\mathcal R (Z_i,u_{j_i},v_{j_i},\lambda,c,\epsilon,\mu,\rho)$ satisfies $C'(\lambda,c,\epsilon,\mu,\rho)$ small cancellation condition as discussed in section \color{blue}4.4\color{black}\ and the map $\gamma_{i+1}$ takes generating set to generating set. In particular $\eta_{i+1}:=\gamma_{i+1}\circ \alpha_{i+1}\circ \beta_{i}$ is surjective homomorphism which takes the generating set of $G_i$ to the generating set of $G_{i+1}$. Let $G^{\mathscr{C}}:=\underset{\rightarrow}{\lim}G_i$. 
We summarize the above discussion in the following statement.
\begin{lemma}\label{limitmaxrankonebyfinite}
	The above construction satisfies the following properties:
	\begin{enumerate}
		\item $G_i$ is non elementary hyperbolic group for all $i$;
		\item Either $u_i\in E(v_i)$ or the group genarated by $\{u_i,v_i \}$ is equal to all of $G_i$;
		\item For every infinite order element $u_k$, $E(u_k)\simeq E^p_{n_0}$ in $G_{n_0}$ for some $ n_0$ and $p$. 
		%\item For each element $x\in X$, $E(x)=E$ in $G_i$, where $x=y^{m^p_1m^p_2\cdots m^p_i}$. The exponent $m^p_i$'s are described as follows: There exist $p\in\mathbb{N}$ such that $x=u_p$. Being a group in ${}_{rk-1}\mathscr{E}_{F}$, $E^p$ can be written as $E^p=\cup_{i=0}^{\infty}E^p_i$, where $E^p_i=\langle g_i\rangle_{\infty}\cup a^p_1\langle g^p_i\rangle_{\infty}\cup a^p_2\langle g^p_i\rangle_{\infty}\cup\cdots\cup a^p_{n_i}\langle g^p_i\rangle_{\infty}$ and $g^p_i=({g^p_{i+1}})^{m^p_{i+1}}$ for some $m^p_{i+1}\in \mathbb{N}$. In other words generator of maximal cyclic group in an elementary group is a power of the generator of maximal infinite cyclic subgroup of the consecutive elementary group in the increasing chain of elementary groups. 
	\end{enumerate}
\end{lemma}
\begin{proof} Part 1. follows from \cite[Lemma 7.2]{Ol93}. To see part 2. notice that by definition if $j_i>i$ then $v_i\in E(u_i)$ in $G_i$. Otherwise if $j_i=i$ then $v_i\notin E(u_i)$ in $G_i$ and $G_i=gp\{u_i,v_i\}$.% Finally, 3. follows immediately from the fact that $x$ is not a proper power in $G_0$.  
\par
For every infinite order element $u_k$ note that $E(u_k)\simeq E^p_{n_0}$ for some $p$ and $n_0$ in $H_{n_0}$ for some $n_0$ by \color{blue}(\ref{amalgumrankonebyfinite})\color{black}. Then by property $(5)$ of \cite[Theorem 2]{Ol93} and Theorem \ref{osin} $E(u_k)\simeq E^p_{n_0}$ in $G_{n_0}$.   
\end{proof}
\begin{remark}\label{lemma4.41}
	\begin{enumerate}
	 \item By part 3. of Lemma  we get that $E(u_k)\simeq E^p_{n_0}$ in $G_{n_0}$ for some $p$ and $n_0$. Hence by \color{blue}(\ref{amalgumrankonebyfinite})\color{black}\ we have $E(u_k)\simeq E^p_{n_0+1}$ in $H_{n_0+1}$. Then by property $(5)$ of \cite[Theorem 2]{Ol93} and Theorem \ref{osin} $E(u_k)\simeq E^p_{n_0+1}$ in $G_{n_0+1}$. By applying same process we get that $E(u_k)\simeq E^p_{n_0+i}$ in $G_{n_0+i}$ for every $i\geq 0$. Hence $E^{\El}(u_k)=\underset{\rightarrow}{\lim}E_l(u_k)=\underset{\rightarrow}{\lim} E^p_{n_0+i}=E^p$, where $E_l(u_k)$ is the maximal elementary subgroup containing the infinite order element $u_k$. 
     \item Note that by part (7) of \cite[Theorem 2]{Ol93} and Theorem \ref{osin} we get that if $x\in G_k$ has finite order then $x\in E^p_i$ for some $p,i$ in $G_k$. 
    \end{enumerate}
\end{remark}
%\begin{remark}
%	Throughout this construction we choose such presentation for the given rank 1 abelian group $L$ and fix that throughout the section.
%\end{remark}
By above construction it is clear that $\{E^{\El}(g)|\ g\in(G^{\mathscr{C}})^0 \}=\mathscr{C}$. 
Suppose that $P$ be a proper maximal subgroup of the group $G^{\mathscr{C}}$. Let $x\in P$ is of finite order. Then by part 2. of remark \color{blue}\ref{lemma4.41}\color{black}\ $x\in E^p$ for some $p\in\mathbb{N}$. If $y\in P\setminus E^p$ then $yg^p_1$ is of infinite order and not in $E^p$. So by construction we get $G=gp\{E^p,yg^p \}\leq P\leq G$. Hence $P=E^p$. Similarly for $x\in P$ has infinite order one can show $P=E^p$ for some $p$ and hence $\mathscr{C}$ is the class of maximal proper subgroups of $G^{\mathscr{C}}$.   	
\end{proof}
\begin{remark} 
	The group $G^{\mathscr{C}}$ in Theorem \color{blue}\ref{rk1byfinite}\color{black}\ also has following properties:
\begin{enumerate} 
	\item Every finitely generated subgroup of the group $G^{\mathscr{C}}$ is either $G^{\mathscr{C}}$ or finite or elementary.
	\item $G^{\mathscr{C}}$ has no finite presentation.
	\item Every finitely presented subgroup of $G^{\mathscr{C}}$ is hyperbolic (being a finitely presented subgroup of a lacunary hyperbolic group) and hence is either finite or elementary. Therefore every finitely presented subgroup of $G^{\mathscr{C}}$ is amenable.
	\item If one starts with a hyperbolic group without the unique product property then it is possible to make $G^{\mathscr{C}}$ a lacunary hyperbolic group without the unique product property.    
 \end{enumerate}  
\end{remark}

\begin{corollary}\label{rankbyfinitemonster}
	Let $E\in {}_{rk-1}\mathscr{E}_{F}$ be a group and $G$ be a torsion free non elementary hyperbolic group. Then there exists a lacunary hyperbolic quotient $\overline{G}$ of $G$ such that every proper maximal subgroup of $G$ is isomorphic to $E$.
\end{corollary}

%%%%%%%%%%%%%%%%%%%%%%%%%%%%%%%%%%%%%%%%%%%%%%%%%%%%%%%%%%%%%%%%%%%%%%%%%%%%%%%%%%%%%%%%%%%%%%%%%%%%%%%%%%%%%%%%%%%%%%%%%%%%%%%%%%%%%%%%%%%%%%%%%%%%%%%%%%%%%%%%%%%%%%%%%%%%%%%%%%%%%%%

\section{Rips type construction}\label{ripssection}
In this section we explore the construction of small cancellation group from any finitely presented group developed by E. Rips \cite{Rip82}.  

\begin{proposition}[\cite{Rip82}]
	Suppose $Q$ is a finitely presented group. Then there exists groups $G$ and  $K$ , for which we get a short exact sequence;
	$$1\rightarrow K\rightarrow G\rightarrow Q\rightarrow 1$$
	such that;
	\begin{enumerate} 
	\item  $G$ is a hyperbolic group.
	\item $K$ is a $2$ generated group.
\end{enumerate}
\end{proposition}

Various types of Rips constructions have been studied in order to construct powerful pathological examples in geometric group theory, see \cite{AS14,BO06,OW04,W03}. General theme of Rips construction is to study exotic properties of normal subgroups of hyperbolic groups by allowing to construct groups with certain group theoretic/geometric properties from a countable group.

\begin{theorem}\label{Rips}

	Let $Q$ is a finitely generated group and $H$ be a non elementary hyperbolic group. Then there exists groups $G$ and $N$ , and a short exact sequence;
	\begin{align} 
	1\rightarrow N\rightarrow G\rightarrow Q\rightarrow 1\nonumber 
	\end{align} 
	such that;
	\begin{enumerate} 
	\item $G$ is a torsion-free lacunary hyperbolic group.
	\item $N$ is a $2$ generated non elementary quotient of $H$.
	\item If $H$ is torsion free then so are $G$ and $N$.
\end{enumerate}
\end{theorem}
\begin{remark}
	Note that $G$ cannot be hyperbolic if $Q$ has no finite presentation.
\end{remark}
\begin{proof}[\underline{\textbf{proof of Theorem \ \color{blue}\ref{Rips}\color{black}}}] 
	If $H$ is not $2$ generated then consider a finite generating set $H=\la X\ra=\{x_1,x_2,\cdots,x_l\}$ with $l\geq 3$ such that $x_i$ has infinite order for all $l\geq i\geq 1$. Note that one can choose such generating set since for every torsion element $s$ in a non elementary hyperbolic group $H$, there exists an infinite order element $t$ such that $st^m$ has infinite order for some $m\in\mathbb{N}$. Without loss of generality assume that $x_1,x_2$ are non commensurable and $E(x_1)\cap E(x_2)=\{1\}$. Note that the subgroup $gp\{x_1,x_2\}$, generated by $x_1,x_2$, is a suitable subgroup of $H$. Applying Theorem \color{blue}\ref{osin}\color{black}\ with non elementary hyperbolic group $H$, suitable subgroup $gp\{x_1,x_2\}$ and $T=\{x_3,x_4,\cdots,x_l \}$ we get that there exists elements $v_3,v_4,\cdots,v_l\in gp\{x_1,x_2\}$ such that the quotient group $\overline{H}:=H/\la\la x_3v_3,x_4v_4,\cdots,x_lv_l\ra\ra $ is non elementary hyperbolic. Note that the quotient map is surjective on $gp\{x_1,x_2\}$ by choice of the set $T$. Hence if $H$ is not $2$ generated we replace $H$ by the quotient $gp\{x_1,x_2\}$ as described above which is a $2$ generated non elementary hyperbolic quotient of the group $H$. Without loss of generality we now assume that $H$ is $2$ generated i.e, $H=\la X=\{x,y\}\ra, \ |X|=2 $ and $x\notin E(y)$.\par   
	Consider a presentation of $Q$.
	\begin{align} 
	Q=\la g_1,g_2,\cdots,g_n|r_1,r_2,r_3,\cdots\ra\nonumber
	\end{align} 
	We assume that $\| r_i\|_{\mathcal{F}(X)} \geq \| r_{i-1}\|_{\mathcal{F}(X)} \ \ \forall \ i\in \mathbb{N}$. where $X=\{g_1,g_2,\cdots,g_n \}$ and $\mathcal{F}(X)$ is free group generated by $X$. Let
	\begin{align} 
	Q_i:=\langle g_1,g_2,\cdots,g_n|r_1,r_2,\cdots,r_i\rangle\nonumber 
	\end{align} 
		
	First consider $G_0:=H*\mathcal{F}(g_1,g_2,\cdots,g_n)$, the free product of non elementary hyperbolic group $H$ with the free group generated by generators of $Q$. Note that $G_0$ is a hyperbolic group. Since $H$ is non elementary hyperbolic group, $H$ is a suitable subgroup of $G_0$. Let $T=\{{g_j}^{-1}x{g_j},g_jx{g_j}^{-1}\ | \ 1\leq j\leq n, \ x\in X \}$. Now by Theorem \color{blue}\ref{osin}\color{black}\ there are elements $\{w_1,w_2,\cdots \}\subset H$ such that the quotient $\overline{G_0}$ is hyperbolic and image of $H$, say $H_0$ is suitable in $\overline{G_0}$. Hence there exists two elements $a,b\in H_0$ such that $E(a)\cap E(b)=\{1\}$.\par 
	Define following set of words as defined in Section \color{blue}4.4\color{black},
	\begin{align}
	R_1:=z_1U^{m_{1,1}}VU^{m_{1,2}}V\cdots VU^{m_{1,j_1}} 
	\end{align}
	where $U,V$ are geodesic representatives of $a,b$ respectively in the group $G_0$ and $z_1$ is a geodesic representative of $r_1$ in $\overline{G_0}$. Let $\mathcal{R}_1$ be the set of all cyclic shifts of $R_1$. Let $N_1$ be a positive number. For any $\epsilon\geq 0,\mu>0,\rho>0$ with $\lambda,c$ given by Lemma \color{blue}\ref{smallwords}\color{black}\ and $\lambda,\mu,\epsilon,\rho,N_1,c$ satisfies condition for Theorem \color{blue}\ref{ol93lemma}\color{black}\ for the hyperbolic group $\overline{G_0}$. Combining the fact that $\mathcal{R}_1$ satisfies $C'(\lambda,c,\epsilon,\mu,\rho)$ and Theorem \color{blue}\ref{ol93lemma}\color{black}\ with $H_0:=gp\{a,b\},H_1:=gp\{a,z_1\}$ and $H_2:=\{b,z_1 \}$, we get that the factor group $G_1:=\langle H|\mathcal{R}_1\rangle $ is hyperbolic and the injective radius of the factor homomorphism is $\geq N_1$. Moreover we get that image of $H_0,H_1,H_2$ are non elementary in the factor group $G_1$, i.e. $a,b\notin E(z_i),a\notin E(b)$ in $G_1$. \par 
	Define $K_1$ to be the images of $H_0$ in $G_1$. We continue by starting with $G_1$ instead of $G_0$ and add relations of the form 
	\begin{align}
	R_i:=z_iU^{m_{i,1}}VU^{m_{i,2}}V\cdots VU^{m_{i,j_i}} 
	\end{align}
	where $z_i=r_i$ for $i\geq 2$, in the hyperbolic group $G_{i-1}$. 
	Hence by induction we get the required group $G$ as a limit of hyperbolic groups $G_i$ and by choosing large enough $N_i$ in every step one can ensure that $G$ is lacunary hyperbolic. Also we get $N$ as inductive limit of groups $K_i$.  \par 
	Let $Q$ has no finite presentation and $G$ is hyperbolic. Then $G$ has a finite presentation and the kernel is finitely generated which implies that $Q$ has a finite presentation. This contradicts the hypothesis that $Q$ has no finite presentation. Hence $G$ is not hyperbolic when $Q$ has no finite presentation. Note that being a non elementary image of a $2$ generated group, $N$ is $2$ generated.  
\end{proof}

As a corollary of Theorem \color{blue}\ref{Rips}\color{black}\ together with the existence of hyperbolic property $(T)$ group without unique product property, one can get following corollary:

\begin{corollary}[Theorem 3,\cite{AS14}]\label{uniqueproduct}
	Let $Q$ be a finitely generated group. Then there exists a short exact
	sequence
	\begin{align}
	1\rightarrow N\rightarrow G\rightarrow Q\rightarrow 1\nonumber 
	\end{align}

	such that;
	\begin{itemize}
		\item $G$ is a torsion-free group without the unique product property which is a direct limit of Gromov hyperbolic groups. 
		\item $N$ is a subgroup of $G$ with Kazhdan’s Property $(T)$ and without the unique product property.
	\end{itemize}
\end{corollary}
\begin{proof}
	In Theorem \color{blue}\ref{Rips}\color{black}, let $H$ to be a hyperbolic property $(T)$ group without unique product property and in the proof one can choose $N_1$ and the sizes of $w_1,w_2,\cdots$ to be greater than the size of the finite sets $A,B$ and $AB$ for which unique product does not hold in the group $H$. Note that one can choose sizes of $w_1,w_2,\cdots$ to be arbitrary large by Remark \color{blue}\ref{osinlacunar}\color{black}\ and also one can choose $N_1$ to be as large as one wants. 
\end{proof}
Note that the group $G$ in Corollary \color{blue}\ref{uniqueproduct}\color{black}\  is a lacunary hyperbolic group which is a special class of direct limits of Gromov hyperbolic groups.

%%%%%%%%%%%%%%%%%%%%%%%%%%%%%%%%%%%%%%%%%%%%%%%%%%%%%%%%%%%%%%%%%%%%%%%%%%%%%%%%%%%%%%%%%%%%%%%%%%%%%%%%%%%%%%%%%%%%%%%%%%%%%%%%%%%%%%%%%%%%%%%%%%%%%%%%%%%%%%%%%%%%%%%%%%%%%%%%%%%%%%%%%%%%%%%%%%%%%%%%%%%%%%%%%%%%%%%%%%%%%%%%%%%%%%%%%%%%%%%%%%%%%%%%%%%%%%%%%%%%%%%%%%%%%%%%%%%%%%%%%%%%%%%%%%%%%%%%%%%%%%%%%%%%%%%%%%%%%%%%%

\section{Applications to von Neumann algebras}\label{maximalvonneumann}

If $\mathcal{M}$ is a von Neumann algebra then a von Neumann subalgebra $\mathcal{N}\subset\mathcal{M}$ is called \emph{maximal} if there is no intermediate von Neumann subalgebra $\mathcal{N}\nsubseteq\mathcal{P}\subsetneq\mathcal{M}$. Theorem \color{blue}\ref{monster}\color{black}\ generalizes Theorem \color{blue}3.10\color{black}\ of \cite{CDK19} and hence we enlarge the class of property $(T)$ groups introduced in \cite{CDK19} that give rise to property $(T)$ von Neumann algebras which have maximal von Neumann subalgebras without property $(T)$. For readers' convenience, we provides short details and notations that were used in \cite[Section 3 \& 4]{CDK19}.\par 
First we are going to state one of the main ingredient for our purpose, the Rips construction due to I.  Belegradek and D. Osin.
\begin{theorem}\cite{BO06}\label{beleosinripsconstr}
	Let $H$ be a non-elementary hyperbolic group, $Q$ be a finitely generated group and $S$ a subgroup of $Q$. Suppose $Q$ is finitely presented with respect to $S$. Then there exists a short exact sequence
	$$1\rightarrow N\rightarrow G\overset{\epsilon}{\rightarrow} Q\rightarrow 1,$$
	and an embedding $\iota:Q\rightarrow G$ such that
	\begin{enumerate}
		\item $N$ is isomorphic to a quotient of $H$. 
		\item $G$ is hyperbolic relative to the proper subgroup $\iota(S)$.
		\item $\iota\circ \epsilon=Id$.
		\item If $H$ and $Q$ are torsion free then so is $G$.
		\item The canonical map $\phi :Q\hookrightarrow Out(N)$ is injective and $[Out(N):\phi(Q)]<\infty$.
	\end{enumerate}
\end{theorem} 
In our setting $H$ is torsion free and has property (T) and $Q=S$ and it is torsion free. In this situation Theorem \color{blue}\ref{beleosinripsconstr}\color{black}\ implies that $G$ is admits a semidirect product decomposition  $G= N\rtimes Q$ and it is hyperbolic relative to $\{Q\}$. We now state another key lemma 
\begin{lemma}[Lemma 4.2, \cite{CDK19}] \label{trivrel}
	Let $N$ be an icc group and let $Q$ be a group together with an outer action $Q \ca^{\sigma} N$. Then $\mathcal L(N)' \cap \mathcal L(N \rtimes_{\sigma}Q)= \mathbb{C}$.
\end{lemma}
\begin{notation}\label{ripsvn} Consider the lacunary hyperbolic groups $Q$ from Theorem \color{blue}\ref{monster}\color{black}\ together with the collection of maximal rank one subgroups $\mathscr{F}:=\{Q^i_m< Q\}_i$. Also let $N\rtimes Q\in \mathcal Rip (Q)$ be the semidirect product obtained via the Rips construction together with the subgroups $N\rtimes Q^i_m<N\rtimes Q$. Throughout this section we will consider the corresponding von Neumann algebras $\mathcal M^i_m :=\mathcal L(N\rtimes Q^i_m)\subset \mathcal L(N\rtimes Q):=\mathcal M$.  
\end{notation}

Assuming Notation \color{blue}\ref{ripsvn}\color{black}, we now show following:
\begin{theorem}\label{maximalvN}  $\mathcal M^i_m$ is a maximal von Neumann algebra of $\mathcal M$ for every $i$. In particular, when $N\rtimes Q\in \mathcal Rip_{\mathcal T}(Q)$  then for every $i$, $\mathcal M^i_m$ is a non-property (T) maximal von Neumann subalgebra of a property (T) von Neumann algebra $\mathcal M$.   
\end{theorem}

\begin{proof} The proof goes along the same line of the proof of \cite[Theorem 4.4]{CDK19} Fix $\mathcal P$ be any intermediate subalgebra $\mathcal M^i_m\subseteq \mathcal P\subseteq \mathcal M$ for some $i$. Since $\mathcal M^i_m \subset \mathcal M$ is spatially isomorphic to the crossed product inclusion $\mathcal L(N)\rtimes Q^i_m \subset \mathcal L(N)\rtimes Q$ we have $\mathcal L(N)\rtimes Q^i_m \subseteq \mathcal P \subseteq \mathcal L(N)\rtimes Q$. By Lemma \ref{trivrel} we have that $(\mathcal L(N)\rtimes Q^i_m)' \cap (\mathcal L(N)\rtimes Q ) \subseteq \mathcal L(N)' \cap ( \mathcal L (N)\rtimes Q)=\mathbb{C}$. In particular, $\mathcal P$ is a factor. Moreover, by the Galois correspondence theorem \cite{Ch78} (see also \cite[Corollary 3.8]{CD19}) there is a subgroup $Q^i_m\leqslant K\leqslant Q$ so that $\mathcal P= \mathcal L(N)\rtimes K$. However since by construction, $Q^i_m$ is a maximal subgroup of $Q$ for every $i$, we must have that $K=Q^i_m$ or $Q$. Thus we get that $\mathcal P = \mathcal M^i_m$ or $\mathcal M$ and the conclusion follows. \par 
	For the remaining part note by \cite{CJ} that $\mathcal M$ has property (T). Also, since $N \rtimes Q^i_m$ surjects onto an infinite abelian group then it does not have property (T). Thus by \cite{CJ} again $\mathcal M^i_m =\mathcal L(N\rtimes Q^i_m)$ does not have property (T) either.
\end{proof}

\section*{Acknowledgements}
The author would like to thank Prof. Alexander Ol'shanskii and Prof. Jesse Peterson for fruitful conversations and their constant support. The author is grateful to Prof. Ionut Chifan and Sayan Das for having stimulating conversation regarding this project.

\noindent
\textsc{Department of Mathematics, Vanderbilt University, 1326 Stevenson Center, Nashville, TN 37240, U.S.A.}\\
\email{krishnendu.khan@vanderbilt.edu}\\
%\email{krishnendukhan.math@gmail.com}

\end{document}